\RequirePackage{fix-cm}
\documentclass[smallextended]{svjour3}
\smartqed

\usepackage{preamble}

\begin{document}

\title{
Distributionally Risk-Receptive and Robust Multistage Stochastic Integer Programs and Interdiction Models 
\thanks{This research is partially funded by National Science Foundation Grant CMMI-1824897 and 2034503, and Commonwealth Cyber Initiative grants which are gratefully acknowledged.}}

\titlerunning{Distributionally Risk-Receptive and Robust Multistage Stochastic Programs}

\author{Sumin Kang \and
        Manish Bansal}
\institute{Sumin Kang \and Manish Bansal \at
    Grado Department of Industrial and Systems Engineering, Virginia Tech, Blacksburg, VA 24061, USA \\
    \email{suminkang@vt.edu} \and \email{bansal@vt.edu}
}

\date{First Version: June 2023; Latest Version: Sep 2024}
\def\makeheadbox{} 

\maketitle

\begin{abstract}
In this paper, we study \underline{d}istributionally \underline{r}isk-\underline{r}eceptive and \underline{d}istribu\-tionally \underline{ro}bust (or risk-averse) multistage stochastic mixed-integer programs (denoted by DRR- and DRO-MSIPs). 
We present cutting plane-based and reformulation-based approaches for solving DRR- and DRO-MSIPs without and with decision-dependent uncertainty to optimality. We show that these approaches are finitely convergent with probability one. 
Furthermore, we introduce generalizations of DRR- and DRO-MSIPs by presenting multistage stochastic disjunctive programs and algorithms for solving them. 
These frameworks are useful for optimization problems under uncertainty where the focus is on analyzing outcomes based on multiple decision-makers' differing perspectives, such as interdiction problems that are attacker-defender games having non-cooperative players. 
To assess the performance of the algorithms for DRR- and DRO-MSIPs, we consider instances of distributionally ambiguous multistage maximum flow and facility location interdiction problems that are important in their own right. Based on our computational results, we observe that the cutting plane-based approaches are 2800\% and 2410\% (on average) faster than the reformulation-based approaches for the foregoing instances with distributional risk-aversion and risk-receptiveness, respectively. Additionally, we conducted out-of-sample tests to showcase the significance of the DRR framework in revealing network vulnerabilities and also in mitigating the impact of data corruption.

\end{abstract}

\maketitle

\section{Introduction} \label{sec:introduction}

Stochastic programming is a framework for modeling optimization problems with uncertain input data, assuming that the probability distribution associated with the uncertain parameters is known. Typically, it employs a probability distribution that approximates the \textit{true} probability distribution (which is often not known in practice) using prior knowledge or historical data. When the complete information of the distribution is not known, distributionally robust optimization provides a framework to a decision-maker for modeling risk-aversion to such \textit{distributional ambiguity} \cite{Scarf}.
However, optimization problems often involve multiple decision-makers, and the risk-aversion modeled by distributionally robust optimization may not be sufficient when the focus is on analyzing outcomes from their diverse perspectives.

For an example, consider a two-player Stackelberg game played between an attacker and a defender with distributional ambiguity. The attacker with the aim of degrading the defender's performance makes their decisions before the defender, thereby influencing the defender's subsequent decision-making process. 
These games are often used to analyze vulnerabilities in the defender's system/network (refer to \cite{Brown,Cormican,Smith,Kang} for specific examples). In this adversarial context, considering the attacker as risk-averse would represent an optimistic view of the defender, downplaying the successfully-attacked scenarios instead. Indeed, this underscores the necessity of a diametrically opposed framework, where the defender evaluates the optimistic view of the attacker as well.
We characterize this risk behavior as \textit{\underline{d}istributionally \underline{r}isk-\underline{r}eceptiveness} (DRR), in contrast with the conventional notion of distributionally risk-aversion \cite{Kang}. 
\blue{Despite this juxtaposition, we adhere to the community convention by using the term \textit{\underline{d}istributional \underline{ro}bustness (or robust)} (DRO) interchangeably with distributional risk-aversion.}

\blue{The DRR framework is also applicable when given samples may be contaminated. This optimistic optimization approach first rectifies the contamination and then makes decision.} We note that concepts similar to DRR have also been explored to mitigate the impact of corrupted data in data-driven decision-making processes (namely, robust statistics~\cite{Blanchet,BlanchetTransport,Jiang} and Rockafellian relaxations~\cite{Royset}; refer to Sect.~\ref{sec:literature_review} for details). \blue{In the context of the Stackelberg games for disrupting domestic sex trafficking networks~\cite{Kosmas} and illegal drug supply chain networks~\cite{Malaviya} where attacker (or interdictor) is the protagonist, the DRR framework allows a risk-receptive interdictor to mitigate the impact of a network user (or defender) who can deceive
the interdictor by providing false information about critical arcs of the network for maintaining the maximum flow in the network (refer to Sect.~\ref{sec:max-flow-results} for details). }

In this paper, we study DRR and DRO models within the framework of multistage stochastic programming. 
A single-stage DRR model is formulated as a min-min problem. The outer problem seeks a decision that minimizes the expected cost for a best-case probability distribution, selected by the inner problem from a set of possible probability distributions, referred to as \textit{ambiguity set}. Conversely, a single-stage DRO model is formulated as a min-max problem, where the inner problem selects a worst-case probability distribution from an ambiguity set.
We extend these models to multistage stochastic integer programs (MSIPs), where a sequence of decisions is made over a time horizon. The objective of each stage is to minimize the sum of the current cost and the expected cost of future stages. We introduce DRR-MSIP and DRO-MSIP models to address distributional ambiguity in MSIPs. Furthermore, we study DRR-MSIPs with decision-dependent uncertainty as decisions often affect uncertainty in many real-world situations such as facility location~\cite{basciftci_ddd_2021,Yu} and machine scheduling \cite{noyan_2022}. For example, in an agricultural supply chain, a farmer would sign up to transport the harvest to the storage centers, only if there is a nearby storage center that has been assigned to a processing mill with adequate demand to consume the farmer's harvest. This implies that decisions such as selection of storage centers and their assignment to the processing mill also impact the uncertain amount of harvest available at the storage centers and hence, the probability distribution of this uncertain parameter.

Despite the simplicity of the DRR models, they face significant computational challenges in terms of solution approaches. Directly solving a DRR model as a single-level minimization is impractical due to the integral defining the expectation, unless the ambiguity set is specifically chosen. With a finite support, the expectation is expressed as a summation, but this transforms the model without decision-dependent uncertainty into a mixed-integer nonconvex bilinear program, which is notoriously difficult to solve. Even without any integrality restrictions and in two-stage setup, the objective function of a DRR model is nonconvex. In a multistage stochastic decision-making setup, the complexity increases further due to factors such as the number of scenarios and stages, which expand the model size. In contrast, some solution approaches for DRO models that utilize approximations are available in the literature under certain conditions on the cost functions and feasible sets \cite{Bansal,BansalJOGO,Duque_ts_2022,Duque,GangaBansal,Philpott18,Yu}. It is worth noting that if the ambiguity sets are singleton, then DRO- and DRR-MSIPs reduce to a (risk-neutral) MSIP.

In this paper, we first address DRR- and DRO-MSIPs without decision-dependent uncertainty by presenting convex lower-approximations of the expected cost-to-go functions for best-case or worst-case distributions. We derive valid cutting planes and utilize reformulation techniques, considering finite or continuous supports. We note that while reformulations for DRR and DRO-MSIPs are derived using standard techniques, they are presented in this paper for comparative purposes as no other approaches for solving DRR-MSIPs are known in the literature. These approximations enable the application of decomposition methods, which effectively reduce complexity in a multistage decision-making setup. Specifically, we introduce customized SDDP (Stochastic Dual Dynamic Programming) algorithms \cite{Pereira} for DRR- and DRO-MSIPs that utilize these approximations. We also generalize these algorithms to solve DRR and DRO multistage stochastic disjunctive linear programs (DRR- and DRO-MSDPs) where the feasible sets are defined by disjunctive linear sets.

\blue{DRO problems where the ambiguity sets are} dependent on the decisions have been addressed by deriving their tractable reformulations \cite{Luo,Yu} and then using state-of-the-art approaches for solving these reformulations. However, a solution approach for \blue{DRR (or optimistic) optimization} with decision-dependent ambiguity set has not been studied in the literature, to the best of our knowledge. We fill this gap by introducing a reformulation-based approximation for decision-dependent DRR-MSIPs where the ambiguity sets are defined using Wasserstein distance. We also present lower-bound approximation for the associated decision-dependent optimistic ``cost-to-go-function" and then utilize it to introduce a cutting plane-based algorithm for DRR-MSIPs with decision-dependent uncertainty. 

\blue{
To demonstrate practical applications of DRR- and DRO-MSIPs, we consider multistage network interdiction problems (NIPs). Specifically, we introduce multistage stochastic models for the maximum flow interdiction problem and the facility location interdiction problem, denoted by MS-MFIP and MS-FLIP, respectively. 
MS-MFIP has applications in disrupting domestic sex trafficking networks~\cite{Kosmas} and illegal drug supply chain networks~\cite{Malaviya}, while MS-FLIP is applicable for identifying critical assets in infrastructure systems~\cite{Church04}. Despite these applications, their multistage stochastic variants have not been studied in the literature. To the best of our knowledge, only two-stage variants where each stage corresponds to the decision-making process of a single player have been explored in \cite{Cormican,Janjarassuk,Sadana}.

Through our DRO and DRR frameworks, we investigate the significance of varying risk-appetite of the decision-maker in MS-MFIP. Also, we examine an adversarial setting where data is corrupted, and demonstrate that the DRR framework can enhance the out-of-sample performance under such environments.
Lastly, we provide numerical results for both MS-MFIP and MS-FLIP, showcasing the computational efficiency of the proposed approaches.
}

In the remainder of this section, we present formulations of general DRR- and DRO multistage stochastic programs (MSPs) along with the contributions and organization of this paper.

\paragraph{Notation.} We use $[d]$ to denote set $\{1, \dots, d\}$ for any positive integer $d$.

\subsection{Problem Formulation: DRO-MSP and DRR-MSP}\label{sec:ProbForm}

We first present the Bellman equation of an (risk-neutral) MSP with a planning horizon of $T$ stages and then generalize it to the cases with distributional ambiguity, i.e., DRO- and DRR-MSPs. 
There are two types of decision variables at each stage: \textit{state} decision variables that influence subsequent stages and \textit{local} decision variables specific to the current stage.
Let $x_t$ and $y_t$ be the state and local decision vector for stage $t \in [T]$, respectively. An MSP is formulated as
\begin{equation} \label{eq:generic_msp_main}
    \min_{(x_1, y_1) \in X_1} 
        \bigg\{f_1(x_1, y_1) + \E_{P_2} \big[ Q_{2}(x_1, \bom_2) \big] \bigg\},
\end{equation}
where $\bom_t$ is a random vector that represents uncertain parameters for each stage $t = 2, \dots, T$, and the cost-to-go functions are defined as
\begin{equation} \label{eq:generic_msp_bellman}
    Q_t (x_{t-1}, \bom_t) =
    \min_{(x_t, y_t) \in X_t(x_{t-1}, \bom_t)} 
        \bigg\{f_t(x_t, y_t, \bom_t) + \E_{P_{t+1}} \big[ Q_{t+1}(x_t, \bom_{t+1}) \big]\bigg\},
\end{equation}
for $t =2, \dots, T,$ and $Q_{T+1} = 0$. 
Set $X_1$ denotes the feasible region of the first stage, and set $X_t(x_{t-1}, \bom_t)$ denotes the feasible region of each stage $t\in\{2, \dots, T\}$ that depends on a decision of the previous stage $x_{t-1}$. The random vector $\bom_t$ has a sample space $\Om_t$, which is associated with a probability distribution $P_t$.

Now, to present DRO- and DRR-MSPs, we consider a set $\P_t$ of probability distributions as the ambiguity set for $t\in\{2, \dots, T\}$. Then, the bellman equation form of a DRO-MSP is given by
\begin{equation} \label{eq:dra_model}
    \min_{(x_1, y_1) \in X_1} \bigg\{f_1(x_1, y_1) + \max_{P_2 \in\P_2} \E_{P_2} \big[ Q_2^{RA} (x_1, \bom_2) \big]\bigg\},
\end{equation}
where
\begin{equation} \label{eq:dra_bellman}
\begin{split}
    &Q_t^{RA} (x_{t-1}, \bom_t) = \\
    &\quad\min_{(x_t, y_t) \in X_t(x_{t-1}, \bom_t)} 
        \bigg\{f_t(x_t, y_t, \bom_t) + \max_{P_{t+1} \in\P_{t+1}} \E_{P_{t+1}} \big[ Q_{t+1}^{RA} (x_t, \bom_{t+1}) \big]\bigg\}
\end{split}
\end{equation}
for $t = 2, \dots, T$, and $Q^{RA}_{T+1} = 0$. We refer to the value function $$\Q_{t+1}^{RA}(x_{t}) := \max_{P_{t+1} \in\P_{t+1}} \E_{P_{t+1}} \big[ Q_{t+1}^{RA} (x_t, \bom_{t+1}) \big]$$ as the \textit{pessimistic expected cost-to-go} function.
By minimizing with respect to the probability distribution within the ambiguity set in \eqref{eq:dra_model} and \eqref{eq:dra_bellman}, instead of maximizing, we can formulate a DRR-MSP as follows:
\begin{equation} \label{eq:drr_model}
\begin{split}
    \min_{(x_1, y_1) \in X_1} 
    \bigg\{f_1(x_1, y_1) + \min_{P_2 \in\P_2} \E_{P_2} \big[ Q_2^{RR} (x_1, \bom_2) \big]\bigg\},
\end{split} 
\end{equation}
where
\begin{equation} \label{eq:drr_bellman}
\begin{split}
    &Q_t^{RR} (x_{t-1}, \bom_t) = \\
    &\quad\min_{(x_t, y_t) \in X_t(x_{t-1}, \bom_t)} 
        \bigg\{f_t(x_t, y_t, \bom_t) + \min_{P_{t+1} \in\P_{t+1}} \E_{P_{t+1}} \big[ Q_{t+1}^{RR} (x_t, \bom_{t+1}) \big]\bigg\}
\end{split}
\end{equation}
for $t =2, \dots, T$, and $Q^{RR}_{T+1} = 0$. The value function $$\Q_{t+1}^{RR}(x_{t}) := \min_{P_{t+1} \in\P_{t+1}} \E_{P_{t+1}} \big[ Q_{t+1}^{RR} (x_t, \bom_{t+1}) \big]$$ is referred to as the \textit{optimistic expected cost-to-go} function. For DRR-MSPs with decision-dependent uncertainty, we consider ambiguity sets that depend on decisions, i.e., the optimistic expected cost-to-go functions are defined by
    $$
        \Q_{t+1}^{RR}(x_{t}) = \min_{P_{t+1} \in\P_{t+1}(x_t)} \E_{P_{t+1}} \big[ Q_{t+1}^{RR} (x_t, \bom_{t+1}) \big].
    $$

When the feasible region $X_t(x_{t-1}, \bom_t)$ is defined by a set of linear inequalities and disjunctive constraints, problems \eqref{eq:dra_model}-\eqref{eq:dra_bellman} and \eqref{eq:drr_model}-\eqref{eq:drr_bellman} are referred to as DRO- and DRR-MSDPs, respectively, i.e., 
\begin{equation} \begin{split} \label{eq:disjunctive_set}
    &X_t(x_{t-1}, \bom_t) := \bigg\{ (x_t, y_t) \in \R^{d_x}_+ \times \R^{d_y}_+:  \\
    &\qquad \bigvee_{h \in H_t} \Big( A^h_t(\bom_t) x_t + B^h_t(\bom_t) y_t \geq b^h_t(\bom_t) - C^h_t(\bom_t) x_{t-1} \Big)\bigg\}.
\end{split} \end{equation}
Here, notation $\vee$ is used to denote disjunction (``or'' logical operator). The disjunctive constraints generalize integrality constraints on variables. For example, a binary restriction on a variable, i.e., $x \in \{0,1\}$, is equivalent to a disjunction: $(x=0) \vee (x=1)$. If the disjunctive constraints in \eqref{eq:disjunctive_set} represent linear inequalities with integrality constraints on variables, then the DRO- and DRR-MSDPs reduce to DRO- and DRR-MSIPs.
The main challenges encountered in solving these problems arises from the nonconvexity of feasible regions, caused by logical disjunctions or integer variables, as well as the nonlinearity and discontinuity of objective functions. Additionally, in the DRR problems, each stage problem's objective function is nonconvex even if all variables are continuous.

\subsection{Contributions and Organization of this Paper} 
In Sect.~\ref{sec:literature_review}, we provide a review of previous studies on MSPs with and without distributional ambiguity. The organization of the rest of this paper and our contributions are as follows. 

\begin{itemize}
    \item {\it DRR-MSIPs with finite support.}
    In Sect.~\ref{sec:approx_drr_msip}, we present a class of cutting planes to under-approximate the optimistic expected cost-to-go function $\Q^{RR}_{t+1}$ at stage $t$. For the sake of computational comparisons, we also reformulate DRR-MSIPs by considering the probability distribution at each stage as a decision vector and applying a McCormick envelope. 
    We show that our SDDP-based algorithms using the cutting plane- and reformulation-based approximations are exact and finitely convergent (Sect.~\ref{sec:da_sddp}). 
    Our computational results in Sect.~\ref{sec:computational_results} show that the algorithm using the cutting plane-based approach is 24.1 times faster (on average) than the algorithm using the reformulation approach.\vspace{0.3em}

    \item {\it DRR-MSIPs with continuous support and decision-dependent ambiguity sets.}
    In Sect.~\ref{sec:approx_drr_msip_dd}, we study DRR-MSIPs with Wasserstein ambiguity sets where the radii of the Wasserstein balls depend on the decision variables. We derive a dual formulation of the inner minimization at each stage~\eqref{eq:drr_bellman}, thereby deriving a class of valid cutting planes for the optimistic expected cost-to-go function and a cutting plane-based solution approach.\vspace{0.3em}

    \item {\it DRO-MSIPs.} 
    In Sect.~\ref{sec:approx_dra_msip}, we present a cutting plane-based approximation for DRO-MSIPs with a general family of ambiguity sets and finite supports. Also, we present a reformulation-based approximation using a dual formulation of DRO-MSIPs with Wasserstein ambiguity sets and continuous/finite support. We observe that the former is 28 times (on average) faster than the latter approach for solving DRO-MSIP with finite support.\vspace{0.3em} 

    \item \textit{Out-of-sample Performance.} By conducting out-of-sample tests, we demonstrate the significance of DRR and DRO in the context of interdiction problem, in particular, MS-MFIP. The results in Sect.~\ref{sec:comp_result_out_of_sample1} show that the DRO framework enhances the robustness of decision policies under uncertainty, and the DRR framework enables the identification of network vulnerabilities by giving more weights to unfavorable scenarios for the network user. Furthermore, in Sect.~\ref{sec:comp_result_out_of_sample2} we illustrate out-of-sample performance in an adversarial setting, where the sample data are intentionally corrupted (refer to robust statistics in Sect. \ref{sec:literature_review} for details). The results show that the DRR policies behaves robust to data corruption by decreasing the significance of the corrupted data points. \vspace{0.3em}
    
    \item {\it DRR-MSDPs and DRO-MSDPs.} In Sect.~\ref{sec:extensions_da_msdp}, we introduce algorithms for solving MSDPs, DRR-MSDPs, and DRO-MSDPs by deriving tight extended formulations for parametric disjunctive constraints in each stage. Since MSIPs are special cases of MSDPs, we utilize the foregoing approach for solving MSIPs with(out) distributional ambiguity as well, where a hierarchy of relaxations of the feasible regions is obtained in each iteration.\vspace{0.3em}

    \item {\it MS-MFIP and MS-FLIP.} As mentioned in Sect.~\ref{sec:introduction}, MS-MFIP and MS-FLIP are important interdiction problems in their own right and have not been studied in the literature. We present algorithms to solve these problems and their distributionally ambiguous variants, thereby generalizing results of \cite{Church04,Cormican,Janjarassuk,Malaviya,Kosmas} that study special cases of these problems.
\end{itemize}

\noindent In Sect.~\ref{sec:computational_results}, we present our computational results and concluding remarks in Sect.~\ref{sec:conclusion}. For readers' convenience, we list major abbreviations used in this paper in Appendix~\ref{apx:abbr} and provide all proofs in Appendix~\ref{apx:proofs}.
Throughout the paper, we made the following assumptions:
\begin{assumption} \label{amp:stagewise_ind}
    The random vectors are stage-wise independent, i.e., $\bom_t$ is independent of $\bom_{[t-1]}=(\bom_2, \dots, \bom_{t-1})$ for all $t = 3, \dots, T$.    
\end{assumption}
We note that the stage-wise independence assumption is required for computationally efficient algorithms, but not deriving convex approximations.

\begin{assumption} \label{amp:integrality}
    The state variables $x_t$ are binary for all $t \in [T]$. The feasible sets $X_1$ and $X_t(\cdot, \cdot)$ are defined as mixed-integer sets (and disjunctive sets only in Sect.~\ref{sec:extensions_da_msdp}).
\end{assumption} \vspace{-1em}
\blue{
\begin{remark} \label{rem:binary_expansion}
    For general integer or discrete state variables $x_t$ that are bounded, we can obtain their equivalent binary representation by using the binary expansion. 
\end{remark}
}\vspace{-1em}
\begin{assumption} \label{amp:feasible_region}
    Sets $X_1$ and $X_t(x_{t-1}, \om_t)$, given any $x_{t-1}\in\{0,1\}^{d_x}$ and $\om_t \in \Om_t$, for all $t = 2, \dots, T$, are nonempty and compact. 
\end{assumption}
\begin{assumption} \label{amp:finite_support} 
    The supports $\Om_t$, for all $t = 2, \dots, T$, are finite, i.e., $|\Om_t| < \infty$. Accordingly, let $p_t^i$ be the probability of scenario $\om_t^i$, i.e., $P_t(\bom_t = \om_t^i)$, for $i \in \N_t$, where $\N_t := \{1, \dots, N_t\}$ is the index set associated with the support $\Om_t$, $t = 2, \dots, T$.
\end{assumption}
Throughout the paper, we assume finite supports unless stated otherwise. Specifically, we will relax this assumption and consider continuous supports in Sect.~\ref{sec:approx_drr_msip_dd} and Remark~\ref{rem:dra_continuous} in Sect.~\ref{sec:approx_dra_msip}.
\begin{assumption} \label{amp:linear_objective_function}
    Functions $f_1(x_1, y_1)$ and $f_t(x_t, y_t, \om_t),$ for $\om_t \in \Om_t$ and $t = 2, \dots, T$ are linear.
\end{assumption}

\section{Literature Review} \label{sec:literature_review}

In this section, we examine studies related to the DRR framework. In addition, we review the literature on solution approaches for multistage stochastic linear programs (MSLPs), MSIPs, and their distributionally robust variants.

\subsection{Distributionally Risk-Receptive Programs and Robust Statistics}

In the literature, the DRR framework has been studied in three main contexts: 
(a) making decisions robust to outliers, errors, and adversarial corruptions in a data-driven setting \cite{Blanchet,Royset}, 
(b) analyzing outcomes based on multiple decision-makers' differing perspective and varying risk-appetite of the decision-makers \cite{Kang}, 
and (c) obtaining the bounds on the true expectations \cite{Duchi,CaoGao}.

In the context of robust statistics, Blanchet et al.~\cite{Blanchet} discuss the connection between the DRR framework and robust statistics, which aim to seek a reliable estimator given samples that may be contaminated. They present the DRR framework as an optimistic optimization approach that first rectifies the contamination and then makes decision. In \cite{BlanchetTransport}, as a rectifying-optimizing approach, a DRR program with an ambiguity set defined by an optimal transport distance has been studied. Royset et al.~\cite{Royset} also study a similar optimistic approach based on \textit{Rockafellian relaxations}. Jiang and Xie~\cite{Jiang} show that the DRR framework with a specific selection of ambiguity sets can recover many robust statistics, such as median and the least trimmed squares.

In another direction, DRR programs are utilized in reinforcement learning \cite{SongTrustRegion} to find an optimistic policy and Bayesian statistics \cite{NguyenDROLikelihood} to approximate a likelihood. 
In \cite{Gotoh}, the authors investigates the DRO and DRR frameworks in comparison to a sample average approximation approach for out-of-sample performance. In particular, they show that by solving both DRO and DRR programs, one of their solutions always outperform the sample average approximation solutions in out-of-sample test.

In the context of obtaining bounds, Duchi et al.~\cite{Duchi} employ the DRR and DRO frameworks to construct a confidence interval for the true optimal objective of a stochastic optimization problem, which can be used to determine the size of the ambiguity sets for a given confidence interval size. Similarly, Cao and Gao~\cite{CaoGao} address problems involving covariate data where uncertain parameters belong to a specific uncertainty set. They show that solving robust and optimistic optimization problems yields worst-case and best-case rewards, respectively, thus forming a confidence interval for the true reward. Nakao et al.~\cite{NakaoMDP} consider a partially observable Markov decision process with distributional ambiguity. They solve a DRR model to obtain an upper bound on the true value function of a DRO partial observable Markov decision process.

To the best of our knowledge, the literature lacks a solution approach for DRR programs involving multistage decision-making with integer variables, which is applicable to interdiction problems (as discussed in Section \ref{sec:introduction}).

\subsection{Multistage Stochastic and Distributionally Robust Programs}

\paragraph{Multistage Stochastic Linear Programs.}
To solve MSLPs, Nested Benders Decomposition (NBD) approximates the cost-to-go function $Q_{t+1}$ at each stage $t$ by a piecewise linear convex function using Benders cutting planes, which are constructed by the dual solutions from the problem at the subsequent stage~\cite{BirgeNBD}. For MSLPs under the stage-wise independence assumption, SDDP~\cite{Pereira} is a NBD-like approach that harnesses scenario sampling to mitigate ``curse of dimensionality'' of dynamic programming without losing the (almost surely) finite convergence of the NBD algorithm.
The iteration complexity of SDDP is also examined in \cite{Lan}. 

\paragraph{Distributionally Robust (or Risk-averse) MSLPs.} Recently, Philpott et al.~\cite{Philpott18} consider distributionally robust multistage stochastic linear programs (denoted by DRO-MSLPs) where ambiguity sets are constructed based on $\chi^2$ distance from a reference probability distribution. Their approaches are based on SDDP embedding separation algorithms that compute a worst-case probability distribution for the different reference probability distributions. Note that their problems consider only continuous variables. Duque and Morton~\cite{Duque} present an SDDP-based algorithm for DRO-MSLPs where the ambiguity set is defined using Wasserstein metric. Their approach is based on the dualization of the inner problem for finding a worst-case probability distribution. They present comparison analysis of the results from their algorithm and those from the modified algorithm of \cite{Philpott18} for Wasserstein metric. Park and Bayraksan~\cite{Park} investigate DRO-MSLPs, where the ambiguity set is defined using $\phi$-divergence, for an application to a water allocation problem. They propose a NBD-type algorithm relying on the dual reformulation of the inner problem for finding a worst-case probability distribution.

\paragraph{Multistage Stochastic Integer Programs.} For solving MSIPs with binary state variables, an extension to SDDP, referred to as SDDiP, has been proposed by Zou et al.~\cite{Zou}. SDDiP uses a new class of cutting planes, constructed based on a Lagrangian relaxation where the strong duality holds for the resulting Lagrangian dual, to approximate the cost-to-go function in each stage. They provide the cut conditions under which SDDiP is finitely convergent.
For a more general class of MSIPs, where all decision variables are allowed to be mixed-integer, there are several studies applying scenario-wise decomposition schemes to the deterministic equivalent formulation. For example, Carøe and Schultz~\cite{Caroe} present a branch-and-bound algorithm based on a dual decomposition approach applied to the deterministic equivalent formulation. The approach uses a Lagrangian relaxation of \textit{non-anticipativity} constraints which enforce the scenarios that follow the same history up to stage $t$ to have the same decisions until stage $t$. Lulli and Sen~\cite{Lulli} propose a branch-and-price algorithm, i.e., a branch-and-bound algorithm with column generation, for the same class of MSIPs.
We also note that for multistage stochastic mixed-integer nonlinear programs, Zhang and Sun~\cite{Zhang} present three decomposition algorithms---one based on NBD and another based on SDDP---which rely on a regularization of the expected cost-to-go function, i.e., $\E_{P_{t+1}}[Q_{t+1}(x_t,\bom_{t+1})]$ at stage $t$, and a class of cutting planes called \textit{generalized conjugacy cut} to approximate the function.

\paragraph{Distributionally Robust MSIPs.} Yu and Shen~\cite{Yu} investigate decision dependent DRO-MSIPs, where state variables are binary and the ambiguity set depends on the state decisions made at the previous stage. They consider three types of ambiguity sets constructed based on the decision-dependent moment information (e.g., mean and variance). They propose mixed-integer linear programming and mixed-integer semidefinite programming reformulations of the problems and solve them using SDDiP. Recently, in the dissertation of Nakao~\cite{Nakao}, a dual decomposition approach is presented for a DRO-MSIP where variables can be mixed-integer and the ambiguity sets are defined using Wasserstein metric. The approach applies the dual reformulations to the inner problems over Wasserstein ambiguity sets in a consecutive manner for deriving a monolithic-minimization deterministic equivalent formulation of the DRO-MSIP. Then, they use a Lagrangian relaxation of the non-anticipativity constraints in the deterministic equivalent formulation to derive a Lagrangian dual. 
Bayraksan et al.~\cite{bayraksan_bounds_2024} divide the scenario tree into subgroups to obtain lower bounds and present conditions for selecting the radii of the ambiguity sets of the subgroup problems, thereby leading to another decomposition approach for DRO-MSIPs.

\section{Convex Approximations for DRR-MSIPs having Finite Supports} \label{sec:approx_drr_msip}

In this section, we present two convex approximations of the optimistic expected cost-to-go function $\Q^{RR}_{t+1}(x_t)$ at stage $t$.
We start by considering affine cuts valid for \blue{$Q^{RR}_{t+1}$:
\begin{align} \label{eq:valid_cut}
    Q^{RR}_{t+1}(x_{t}, \om_{t+1}^i) \geq (\alpha^{i,k}_{t})^\top x_{t} + \beta^{i,k}_{t}, \quad \forall x_{t} \in \{0, 1\}^{d_x},
\end{align}
and denote them by their coefficients $(\alpha_{t}^{i,k}, \beta_{t}^{i,k})$, where $i \in \N_{t+1}$ indexes the support $\Om_{t+1}$, and $k \in [K_t]$, with $K_t \geq 1$ represents the number of cuts.}
Throughout this section, we suppose these cuts are given for every stage $t$. Discussions regarding the cut generating procedure and properties of these cuts are deferred to Sect.~\ref{sec:da_sddp}.
\blue{
Using the cuts, we define the following nonconvex (bilinear) approximating problem:
\begin{subequations}
\begin{align} \label{eq:drr_simple_approx}
    \phi^{B}_t(x_t) :=
    \min_{\theta_t, P_{t+1}} \ &
        \sum_{i \in \N_{t+1}} p_{t+1}^i \theta_t^i \\
        \text{s.t.}\ 
        & \theta_t^i \geq (\alpha_{t}^{i,k})^\top x_t + \beta_{t}^{i,k}, \quad \forall k \in [K_t],\ i \in \N_{t+1},\\
        & P_{t+1} \in \P_{t+1}.
\end{align}
\end{subequations}
Note that we consider the probabilities $P_{t+1} = (p^i_{t+1})_{i \in \N_{t+1}}$ as decision variables. This leads to bilinear terms involving variables $p^i_{t+1}$ and $\theta^i_t$ in the objective function.
By construction, the function $\phi^{B}_t(x_t)$ provides a lower bound for the optimistic expected cost-to-go function $\Q^{RR}_{t+1}(x_t)$, yet solving this problem directly is not desirable due to the bilinear terms.
}
In the following sections, we derive convex approximations of this problem using a new class of cutting planes and a mixed-integer linear programming reformulation, respectively. For simplicity of exposition, we let $K_t = 1$ for all $t \in [T-1]$ and suppress the index $k$ in notation.

\subsection{A cutting plane-based approximation for DRR-MSIP} \label{sec:cutting_plane_approximation}
Consider any $\hat{x}_t \in \{0,1\}^{d_x}$. We define a cutting plane-based approximating function as:
\begin{align} \label{eq:drr_decomp}
    \phi^{C}_t(x_t) := 
    \min \Big\{ \phi : 
    \phi \geq \pi_{t}^\top (x_{t} - \hat{x}_{t}) + \gamma_t \Big\}
\end{align}
where the parameters of the inequalities are given by
\begin{equation} \label{eq:drr_decomp_coefficients}
    \pi_{t,j} :=
    \begin{cases}
        \min_{P_{t+1} \in \P_{t+1}} \sum_{i \in \N_{t+1}} p^i_{t+1} \alpha^{i}_{t,j}, \quad \text{{\normalfont if }} \hat{x}_{t,j} = 0, \\
        \max_{P_{t+1} \in \P_{t+1}} \sum_{i \in \N_{t+1}} p^i_{t+1} \alpha^{i}_{t,j}, \quad \text{{\normalfont if }} \hat{x}_{t,j} = 1,
    \end{cases}
    \text{for } j \in [d_x],
\end{equation}
and
$
    \gamma_t := \min_{P_{t+1} \in \P_{t+1}} \sum_{i \in \N_{t+1}} p^i_{t+1} \Big( (\alpha^{i}_t)^\top \hat{x}_t + \beta^{i}_t \Big),
$
for each $t \in [T-1]$.
Apparently, this function is convex. Also, it provides a lower-bound for $\Q^{RR}_{t+1}$ as shown in Theorem~\ref{thm:drr_cutting_plane_approx}.
\begin{theorem} \label{thm:drr_cutting_plane_approx}
    The function $\phi^{C}_t$ provides a lower bound for the optimistic expected cost-to-go function, i.e.,
    $\phi^{C}_t(x_t) \leq \Q^{RR}_{t+1}(x_{t})$ for all $x_{t} \in \{0, 1\}^{d_x}$ and $t \in [T-1]$.
\end{theorem}
An important property of this cut is that it preserves the tightness of the cuts $(\alpha_t^{i}, \beta_t^{i})$, i.e., 
it intersects with $\Q^{RR}_{t+1}$ at $\hat{x}_t$ if the cuts $(\alpha^{i}_t, \beta^{i}_t)$ intersect with $Q^{RR}_{t+1}(\cdot, \om^i_{t+1})$ at $\hat{x}_t$ for all $i \in \N_{t+1}$. This property plays a key role in showing the finite convergence of our algorithm presented in Sect.~\ref{sec:da_sddp}.
\blue{Additionally, it is important to note that the cuts~\eqref{eq:drr_decomp} are different from Benders cuts, which rely on the duality results of problems. Since the DRR problem~\eqref{eq:drr_bellman} is nonconvex even when involving continuous variables, convex combination of Benders cuts cannot be applied to approximate the optimistic expected cost-to-go function~$\Q^{RR}_{t+1}(x_t)$ as it can be done for risk-neutral and pessimistic expected cost-to-go-function. 
}

\subsection{A reformulation-based approximation for DRR-MSIP} \label{sec:drr_sddp_r}

We present a reformulation-based approximation for DRR-MSIP by treating the probability distribution as variables and using McCormick envelopes for bilinear terms. Though straightforward, it remains the only approach using existing techniques to solve this problem and therefore, we include it for comparative purpose. More specifically, it works as follows. 

For each constraint in \eqref{eq:drr_simple_approx}, we multiply $p_{t+1}^i$ to both sides of inequalities for each $i \in \N_{t+1}$, and replace $p_{t+1}^i x_t$ with a decision vector $\eta_{t}^i$ in the right-hand side of the resulting inequalities. This yields the following system of inequalities:
\begin{subequations} \begin{alignat}{2}
    p_{t+1}^i \theta_t^i \geq (\alpha_t^{i})^\top \eta_t^i + \beta_t^{i} p_{t+1}^i, \quad \forall i \in \N_{t+1}, \label{eq:reform_inequalities_a} \\
    \eta_t^i \leq x_t, \ 
    \eta_t^i \leq p_{t+1}^i, \ 
    \eta_t^i \geq p_{t+1}^i + x_t - 1, \ 
    \eta_t^i \geq 0, \quad \forall i \in \N_{t+1}. \label{eq:reform_inequalities_b}
\end{alignat} \end{subequations}
Notice that a system of inequalities~\eqref{eq:reform_inequalities_b} ensure that a feasible $\eta_t^i$ equals to $p_{t+1}^i x_t$, given any $x_t \in \{0,1\}^{d_x}$ and $(p_{t+1}^i)_{i \in \N_{t+1}} \in [0, 1]^{N_{t+1}}$.
We introduce an additional variable $\bar{\theta}_t^i$ to replace $p^i_{t+1} \theta_t^i$.
This yields an approximating function of $\Q^{RR}_{t+1}$ for each $t \in [T-1]$ as follows:
\begin{subequations} \label{eq:drr_reform_approx}
\begin{align}
    \phi^{R}_t(x_t) := 
    \min_{P_{t+1}, \bar{\theta}, \eta_t} \bigg\{ \sum_{i \in \N_{t+1}} \bar{\theta}_i :
        \bar{\theta}_i \geq (\alpha^{i}_t)^\top \eta^i_t + \beta^{i}_t p^i_{t+1}, \ \forall i \in \N_{t+1}, \label{eq:drr_reform_approx_a} \\
        \ \eqref{eq:reform_inequalities_b},
        \ P_{t+1} \in \P_{t+1}. \label{eq:drr_reform_approx_b}
    \bigg\}
\end{align}
\end{subequations}
Since $\theta_t^i$ is not restricted, we can readily show that the equivalence between \eqref{eq:reform_inequalities_a} and constraint \eqref{eq:drr_reform_approx_a}.
Therefore, the function $\phi^{R}_t$ equals to the function $\phi^{B}_t$, which is a lower bound for $\Q^{RR}_{t+1}(x_t)$.
When the ambiguity set~$\P_{t+1}$ is defined by a polytope (e.g., Wasserstein ambiguity set with a finite support), the problem~\eqref{eq:drr_reform_approx} is a linear program. Consequently, $\phi^R_t$ becomes piecewise linear and convex by linear programming duality.

\section{Convex Approximation for DRR-MSIPs having Continuous Supports and Decision-dependent Ambiguity Sets} \label{sec:approx_drr_msip_dd}

In this section, we investigate DRR-MSIPs, where for each stage $t$ the support $\Om_t$ is continuous. Suppose that a finite set of data $\bar{\Om}_t := \{\om_t^1, \dots, \om_t^{N_t}\}$ is available. Consider an empirical distribution $\bar{P}_t = \frac{1}{N_t} \sum_{i \in \N_t} \delta_{\om_t^i}$, where $\delta_{\om_t^i}$ is the Dirac delta function centered at $\om_t^i$ for $i \in \N_t$. 
We now present a dual-based convex approximation of the optimistic expected cost-to-go function $\Q^{RR}_{t+1}(x_t)$ at stage $t$. This approximation can be applied to a further generalized case where ambiguity sets are decision-dependent. In particular, we focus on decision-dependent ambiguity sets defined using Wasserstein metric as follows:
\begin{equation} \label{eq:DD_Wasserstein_ambiguity_sets}
\begin{aligned}
    \P_{t}(x_{t-1}) := \bigg\{
        P_t \in \M(\Om_t) : 
        \W(P_t, \bar{P}_t) \leq \epsilon_t(x_{t-1})
    \bigg\}
\end{aligned}
\end{equation}
for $t = 2, \dots, T$ \cite{Luo}, where $\M(\Om_t)$ is a set of all probability distributions supported on $\Om_t$, and $\W(P_t, \bar{P}_t)$ is the $1$-Wasserstein distance defined as
\begin{equation} \label{eq:Wasserstein_distance}
\begin{aligned}
    \W(P_t, \bar{P}_t) := \inf_{P \in \P(\Om_t \times \Om_t)} \bigg\{
        \E_{\pi} \Big[\norm{\om_t - \om_t'}\Big] : P(\om_t) = P_t, \ P(\om_t') = \bar{P}_t
    \bigg\}.
\end{aligned}
\end{equation}
Here, $\P(\Om_t \times \Om_t)$ is the set of all joint probability distributions supported on $\Om_t \times \Om_t$. The marginal distributions of $\om_t$ and $\om_t'$ are denoted by $P(\om_t)$ and $P(\om_t')$, respectively, and $\norm{\cdot}$ is an arbitrary norm. For a given $\epsilon_t(x_{t-1}) > 0$, each ambiguity set $\P_t(x_{t-1}), t \in \{2, \dots, T\},$ is a Wasserstein ball containing all probability distributions within a certain radius $\epsilon_t(x_{t-1})$ from the empirical probability distribution. For the ease of exposition in this section, we use $x_t$ to represent all decision variables at each stage $t$.
\begin{proposition}[Strong duality] \label{prop:drr_continuous_strong_dual}
For the ambiguity set $\P_t(x_{t-1})$ defined as \eqref{eq:DD_Wasserstein_ambiguity_sets}, the optimistic expected cost-to-go function can be reformulated as
    \begin{equation}
    \begin{split}
        \Q^{RR}_t(x_{t-1}) &= \min_{P_t \in \P_t(x_{t-1})} \E[Q^{RR}_t(x_{t-1}, \bom_t)] \\
        & 
        \begin{aligned}
        =\max_{\rho_t \geq 0} 
        \bigg\{&
            -\epsilon_t(x_{t-1}) \rho_t \\
            & + \sum_{i \in \N_t} \frac{1}{N_t} \min_{\om_t \in \Om_t} 
            \Big\{ 
                \rho_t \norm{\om_t - \om_t^i} + Q^{RR}_t(x_{t-1}, \om_t)
            \Big\}
        \bigg\}.
        \end{aligned}
    \end{split}
    \end{equation}
\end{proposition}
Using Proposition~\ref{prop:drr_continuous_strong_dual}, we derive an approximation for $\Q^{RR}_t$ in Theorem \ref{thm:drr_continuous_cut} under the following assumptions. 
\begin{assumption}
    The support $\Om_t$ for each stage $t$ is a bounded hyperrectangle, i.e., $\Om_t = [l_t, u_t]^{d_\om}$.
\end{assumption}

\begin{assumption}
    Random vector $\bom \in \R^{d_\om}$ is only associated with the right-hand side of the constraints defining $X_t(x_{t-1}, \bom_t)$. Let $\X_t \subseteq \R^{d_x}$ be the set defined by the integrality constraints for $x_t$. Then the feasible sets can be written as
    \begin{equation}
    \begin{aligned}
        X_t(x_{t-1}, \bom_t) = \bigg\{
            x_t \in \X_t: A_t x_t \geq \bom_t - C_t x_{t-1}
        \bigg\}.
    \end{aligned}
    \end{equation}
\end{assumption}
\begin{assumption}
    The decision-dependent radius $\epsilon_t(x_{t-1}) : \X_{t-1} \to \R_+$ is an affine function.
\end{assumption}

\begin{theorem} \label{thm:drr_continuous_cut}
Given valid cuts' coefficients $\{(\pi^k_t, \gamma^k_t)\}_{k \in [K_t]}$ for $\Q_{t+1}^{RR}(x_t)$,
\begin{itemize}
\item[\textit{(a)}] the optimistic expected cost-to-go function $\Q^{RR}_t(x_{t-1})$ is lower-approximated as follows:
\begin{align} \label{eq:drr_continuous_cut}
\begin{split}
    \Q^{RR}_t(x_{t-1}) \geq & - \epsilon_t(x_{t-1}) \rho_t + \sum_{i \in \N_t} \frac{1}{N_t} \bigg( (-\lambda^i_t)^\top C_t x_{t-1}
    + (\mu^i_t)^\top l_t \\ & - (\nu^i_t)^\top u_t + (\lambda^i_t - \mu^i_t + \nu^i_t)^\top \om^i_t + \sum_{k \in [K_t]}\zeta^i_{tk} \gamma^k_t \bigg),
\end{split}
\end{align}
such that parameters $(\rho_t, \{\lambda^i_t, \mu^i_t, \nu^i_t, \zeta^i_t\}_{i \in \mathcal{N}_t})$ satisfy the following system of equations and inequalities \eqref{eq:drr_continuous_feasible}:
\begin{subequations}\label{eq:drr_continuous_feasible}
\begin{align} 
    & A_t^\top \lambda^i_t - \sum_{k \in [K_t]} \zeta^i_{tk} \pi^k_t = c_t, \quad \forall i \in \N_t, \label{eq:drr_continuous_feasible_a}\\
    & \norm{\lambda^i_t + \mu^i_t - \nu^i_t}_\ast \leq \rho_t, \quad \forall i \in \N_t, \label{eq:drr_continuous_dual_approx_norm}\\ 
    & \sum_{k \in [K_t]} \zeta^i_{tk} = 1, \quad \forall i \in \N_t \label{eq:drr_continuous_feasible_c}\\
    & \rho_t \geq 0, \ (\lambda^i_t, \mu^i_t, \nu^i_t, \zeta^i_t) \geq 0, \quad \forall i \in \N_t,\label{eq:drr_continuous_feasible_d}
\end{align}
\end{subequations}
where $\norm{\cdot}_\ast$ is the dual norm of $\norm{\cdot}$.

\item[\textit{(b)}] the strongest lower bound approximation~\eqref{eq:drr_continuous_cut} is obtained by solving the following problem:
\begin{align}\label{eq:drr_continuous_dual_approx}
\begin{split}
    \max & \bigg\{- \epsilon_t(x_{t-1}) \rho_t + \frac{1}{N_t} \sum_{i \in \N_t} \bigg( (-\lambda^i_t)^\top C_t x_{t-1} 
    + (\mu^i_t)^\top l_t \\ & - (\nu^i_t)^\top u_t + (\lambda^i_t - \mu^i_t + \nu^i_t)^\top \om^i_t + \sum_{k \in [K_t]}\zeta^i_{tk} \gamma^k_t \bigg) : \eqref{eq:drr_continuous_feasible_a}\text{-}\eqref{eq:drr_continuous_feasible_d} \bigg\}.
\end{split}
\end{align}
\end{itemize}
\end{theorem}
Problem~\eqref{eq:drr_continuous_dual_approx} is a cut-generating problem, where any feasible solution yields a cut in the form of \eqref{eq:drr_continuous_cut}. With the choice of specific norms, problem~\eqref{eq:drr_continuous_dual_approx} becomes computationally manageable. For instance, if $\norm{\cdot}$ represents the $l_1$ norm, the dual norm is the $l_\infty$ norm, allowing the linearization of constraints~\eqref{eq:drr_continuous_dual_approx_norm} and transformation of  problem~\eqref{eq:drr_continuous_dual_approx} into a linear program. Also, if $\norm{\cdot}$ is the $l_2$ norm, its dual norm is also the $l_2$ norm, resulting in problem~\eqref{eq:drr_continuous_dual_approx} becoming a second-order conic program.

Using the cuts generated as above, we derive a cutting plane-based approximating function for $\Q^{RR}_t$ as in the form of \eqref{eq:drr_decomp}. It is important to note that the cuts for next stage $(t+1)$ in Theorem~\ref{thm:drr_continuous_cut} are naturally given through our algorithmic procedure (discussed in Sect.~\ref{sec:da_sddp}). Our algorithm produces cuts in a descending order of stages, from $T$ down to $1$, and thus there always exists cuts for stage $(t+1)$ when generating a cut at stage $t$.

\section{Convex Approximations for DRA-MSIPs} \label{sec:approx_dra_msip}

In this section, we present two convex approximations of the pessimistic expected cost-to-go functions $\Q^{RA}_{t+1}(x_t)$ for stage $t$. These approximations are constructed using cutting planes derived by a separation approach and a reformulation derived by utilizing strong duality, respectively. 

\subsection{A cutting plane-based approximation for DRA-MSIP} \label{sec:dra_sddp_c}

We assume that the supports are finite (Assumption~\ref{amp:finite_support}), and the valid cuts defined by $(\alpha_t^i, \beta_t^i)$ for $Q^{RA}_{t+1}(\cdot, \om^i_{t+1})$ for $t \in [T-1]$ and $i \in \N_{t+1}$ are available. Given cuts $(\alpha^{i}_{t}, \beta^{i}_{t}),$ for $i \in \N_{t+1},$ and solution $\hat{x}_t$, we identify a worst-case probability distribution by solving the following problem, referred to as \textit{distribution separation problem} at stage $t \in [T-1]$:
\begin{align} \label{eq:dra_dsp}
    \max_{P_{t+1} \in \P_{t+1}} \sum_{i \in \N_{t+1}} p_{t+1}^i \Big( (\alpha^{i}_{t})^\top \hat{x}_{t} + \beta^{i}_{t} \Big).
\end{align}
Let $\hat{P}_{t+1} = (\hat{p}^{1}_{t+1}, \dots, \hat{p}^{N_{t+1}}_{t+1})$ be an optimal solution to the distribution separation problem~\eqref{eq:dra_dsp}. 
We define a cutting plane-based approximating function for $t \in [T-1]$ as
\begin{equation} \label{eq:dra_sep_approx}
    \phi^{S}_t(x_t) := \min \Big\{ \phi : \phi \geq \pi_t^\top x_t + \gamma_t \Big\},
\end{equation}
where 
$
    \pi_t = \sum_{i \in \N_{t+1}} \hat{p}^{i}_{t+1} \alpha^{i}_t,
$
and
$
    \gamma_t = \sum_{i \in \N_{t+1}} \hat{p}^{i}_{t+1} \beta^{i}_t.
$
Clearly, the function $\phi^{S}_t$ under-approximates $\Q^{RA}_{t+1}$, since for solution $\bar{x}_t \neq \hat{x}_t$ it holds that $\Q^{RA}_{t+1}(\bar{x}_t) \geq \max_{P_{t+1} \in \P_{t+1}} \sum_{i \in \N_{t+1}} p_{t+1}^i ( (\alpha^{i}_{t})^\top \bar{x}_{t} + \beta^{i}_{t} ) \geq \phi^S_t(\bar{x}_t)$. Note that there are various types of ambiguity sets for which we can solve problem~\eqref{eq:dra_dsp} and compute the coefficients~\eqref{eq:drr_decomp_coefficients} using a finite-time algorithm; e.g., ambiguity sets constructed using Wasserstein metric \cite{Gao}, moment information \cite{Yu}, total variation distance \cite{Bayraksan} and $\chi^2$ distance \cite{Philpott18}, where their supports are finite.

\subsection{A reformulation-based approximation for DRA-MSIP} \label{sec:dra_sddp_r}

We present a reformulation-based approximation that relies on the dualization of the inner maximization problems in \eqref{eq:dra_model} and \eqref{eq:dra_bellman}, i.e., $\max\{\E_{P_{t+1}}[Q^{RA}_{t+1}(\cdot, \cdot)]: P_{t+1} \in \P_{t+1}\}$ for $t \in [T-1]$. 
We note that similar reformulation-based approaches are studied in recent papers for DRA-MSLPs with $\phi$-divergence-based ambiguity sets \cite{Park}, DRA-MSLPs with Wasserstein ambiguity sets \cite{Duque}, and DRA-MSIPs with moment-based ambiguity sets \cite{Yu}. In the following, we demonstrate the DRA-SDDP-R algorithm for DRA-MSIPs with Wasserstein ambiguity sets, but we note that this approach can be applied to DRA-MSIPs with a general family of ambiguity sets for which dual formulations of DRA models are available.

The Wasserstein ambiguity set with finite support is defined for $t = 2, \dots, T$ as
\begin{equation*}
\begin{aligned}
    \P_{t} = \Bigg\{
        P_t \in \R_+^{N_t} : 
        \sum_{i \in \N_t} p_{t}^i = 1, 
        \ \sum_{j \in \N_t} v_{ij} = p_{t}^i, \ i \in \N_t, 
        \ \sum_{i \in \N_t} v_{ij} = \bar{p}_{t}^j, \ j \in \N_t, \\
        \sum_{i \neq j \in \N_t} {\| \om_t^i - \om_t^j \| v_{ij} \leq \epsilon_t}, 
        \ v_{ij} \geq 0, \ \forall i, j \in \N_t
    \Bigg\},
\end{aligned}
\end{equation*}
where $\{\bar{p}_t^i\}_{i \in \N_t}$ is a reference probability distribution on $\Om_t$ for $t \in \{2, \dots, T\}$.
The dual of the maximization in \eqref{eq:dra_model} and \eqref{eq:dra_bellman} for $t \in [T-1]$ is given as follows: 
\begin{equation*}
\begin{split}
    \min_{\rho_t \geq 0} \bigg\{ & \epsilon_{t+1} \rho_t + \sum_{i \in \N_{t+1}} \bar{p}_{t+1}^i \nu^i_t: \\
    &\nu^i_t + \norm{\om^i_{t+1} - \om^j_{t+1}} \rho_t \geq Q^{RA}_{t+1}(x_t, \om^j_{t+1}), \ \forall i,j\in\N_{t+1}
    \bigg\}.
\end{split}
\end{equation*}
Here, the strong duality holds by Theorem~1 in Gao and Kleywegt~\cite{Gao}. Then, given cuts $(\alpha^{i}_t, \beta^{i}_t), i \in \N_{t+1}$, we define a reformulation-based approximating function for $t \in [T-1]$ as
\begin{subequations} \label{eq:dra_reform_approx}
\begin{align}
    \phi^{D}_t(x_t) := \min\Big\{ & \epsilon_{t+1} \rho_t + \sum_{i \in \N_{t+1}} \bar{p}_{t+1}^i \nu^i_t: \rho_t \geq 0, \\
        & \nu^i_t + \norm{\om^i_{t+1} - \om^j_{t+1}} \rho_t \geq (\alpha^{i}_t)^\top x_t + \beta^{i}_t, \ \forall i, j \in \N_{t+1} \Big\}. \label{eq:dra_reform_approx_cut}
\end{align} 
\end{subequations}

\begin{remark} \label{rem:dra_continuous}
The convex approximation described in this section can be applied to DRA-MSIPs having continuous supports by employing additional algorithmic techniques.
More specifically, the dual of the inner maximization in \eqref{eq:dra_model} and \eqref{eq:dra_bellman} with continuous support is given by the following semi-infinite program:
\begin{equation} \label{eq:dra_continuous_reform_approx}
\begin{split}
    \min_{\rho_t \geq 0} & \bigg\{ \epsilon_{t+1} \rho_t + \sum_{i \in \N_{t+1}} \bar{p}_{t+1}^i \nu^i_t: \\
    &\nu^i_t + \norm{\om^i_{t+1} - \om_{t+1}} \rho_t \geq Q^{RA}_{t+1}(x_t, \om_{t+1}), \ \forall i\in\N_{t+1}, \om_{t+1} \in \Om_{t+1}
    \bigg\}.
\end{split}
\end{equation}
Given $\om^j_{t+1} \in \Om_{t+1}$ and cuts $(\alpha^j_t, \beta^j_t)$ valid for $Q^{RA}_{t+1}(\cdot, \om^j_{t+1}),$ for some $j$, we obtain valid cutting planes in the form of \eqref{eq:dra_reform_approx_cut} for this semi-infinite program. These constraints provide an outer approximation of the feasible region of the semi-infinite program. Consequently, the optimal value within the outer approximation is a lower bound on the value of \eqref{eq:dra_continuous_reform_approx}. We iteratively improve the outer approximation through the following steps:
\textit{(i)} Solve the outer approximation of \eqref{eq:dra_continuous_reform_approx} with a finite number of cutting planes;
\textit{(ii)} Given its solution \((\bar{x}_t, \bar{\rho}_t)\), solve the separation problem \(\max_{\omega_{t+1} \in \Omega_{t+1}} \{ Q^{RA}_{t+1}(\bar{x}_t, \omega_{t+1}) - \|\omega^i_{t+1} - \omega_{t+1}\|\bar{\rho}_t \}\);
\textit{(iii)} Add a cutting plane in the form of \eqref{eq:dra_reform_approx_cut} associated with a scenario identified by solving the separation problem to the outer approximation of \eqref{eq:dra_continuous_reform_approx};
\textit{(iv)} Repeat until the outer approximation is as tight as a predetermined level.
These iterative steps refine the outer approximation, making it increasingly accurate in representing the feasible region of the semi-infinite program, and thereby solve the DRA-MSIP with continuous support.

\end{remark}

\section{SDDP-based Algorithms for DRR-MSP and DRO-MSP} \label{sec:da_sddp}

\subsection{Distributionally ambiguous SDDP}

We present a customized SDDP algorithm, referred to as \textit{distributionally ambiguous SDDP} (DA-SDDP), for both DRO-MSIPs and DRR-MSIPs, where the convex approximations discussed in Sects.~\ref{sec:approx_drr_msip} and \ref{sec:approx_dra_msip} are utilized. The DA-SDDP approximates the pessimistic and optimistic expected cost-to-go functions $\Q_{t+1}^{RA}$ and $\Q_{t+1}^{RR}$ for each $t \in [T-1]$. Notably, DA-SDDP shares the key sampling ideas of SDDP for approximating the expectation functions, but the crux of this algorithm is in deriving convex approximations and iterative refinement methods for them. Recall that SDDP solves the linear programs and use dual solutions to derive a valid cut (or approximations) for MSLPs. However, this is not applicable for DRR programs.

\renewcommand{\algorithmiccomment}[1]{\hfill $\triangleright$ #1}

\begin{algorithm}
\caption{Distributionally Ambiguous SDDP} \label{alg:da_sddp}
\begin{algorithmic}[1]
  \STATE{\textbf{Initialize} 
    $l \gets 1$; $x_0 \gets$ initial state; $\om_1 \gets$ data at the first stage; $\Om_1:=\{\om_1\}$; $K^l_t \gets 0$ for $t \in [T-1]$;
  } \\
  \WHILE{(satisfying none of stopping conditions)} \label{line:da_sddp_while} 
    \STATE{Sample a scenario path $\xi^l \in \Xi:=\Om_1 \times \dots \times \Om_T$} \label{line:da_sddp_sample}
    \FOR[Forward Step]{$t \in [T]$} \label{line:da_sddp_forward_for_t}
        \STATE{Solve subproblem $\Prob^l_t(x^l_{t-1}, \xi^l_t)$ and obtain $(x^{l}_t, y^{l}_t)$ and $\hat{Q}_t^l(x^{l}_{t-1}, \xi^l_t)$} \label{line:da_sddp_solve_subproblem}
    \ENDFOR
    \FOR[Backward Step]{$t=T,\dots,2$\label{line:da_sddp_backward_for_t}}
        \FOR{$i \in \N_t$ \label{line:da_sddp_backward_for_i}}
            \STATE{Solve relaxation $\tilde{\Prob}^l_t(x^l_{t-1}, \om_t^i)$
            and obtain cut $(\alpha_{t-1}^{i,l}, \beta_{t-1}^{i,l})$} \label{line:da_sddp_solve_relaxation}
        \ENDFOR
        \STATE{Refine approximating function $\phi^l_{t-1}$ by using cuts $(\alpha_{t-1}^{i,l}, \beta_{t-1}^{i,l}), i \in \N_{t}$} \label{line:da_sddp_refine_approx}
        \STATE{$K^l_{t-1} \gets K^l_{t-1} + 1$} \label{line:da_sddp_add_K}
    \ENDFOR
    \STATE{Solve subproblem $\Prob^l_1(x^l_{0}, \om_1)$ and obtain the bound $LB$} \label{line:da_sddp_compute_bound}
    \STATE{$K^{l+1}_t \gets K^l_t$ for $t =1, \dots, T-1; \quad l \gets l + 1$}
  \ENDWHILE
  \RETURN Subproblems $\{\Prob^l_t\}_{t \in [T-1]}, LB$ \label{line:da_sddp_return}
\end{algorithmic}
\end{algorithm}
A pseudocode of DA-SDDP is given in Algorithm~\ref{alg:da_sddp}. 
The algorithm is initialized with a predetermined initial state of the model, denoted by $x_0$, the input data for the first stage, denoted by $\om_1$, and the singleton set $\Om_1:=\{\om_1\}$, which are introduced to simplify the notation later. We also set the iteration counter $l$ to $1$, and the number of cuts $K^l_t$ for $t \in [T-1]$ at iteration $l$ to $0$.
At iteration $l$, the algorithm samples a scenario path $\xi^l = (\xi^l_1, \cdots, \xi^l_T)$ from $\Xi := \Om_1 \times \dots \times \Om_T$ (Line~\ref{line:da_sddp_sample}). For the finite convergence of the algorithm, we assume that the scenario path sampling is conducted with replacement. Note that this can be readily extended to the sampling of multiple scenario paths per iteration. The remainder of the iteration consists of a forward step (Lines~\ref{line:da_sddp_forward_for_t} and \ref{line:da_sddp_solve_subproblem}) and a backward step (Lines~\ref{line:da_sddp_backward_for_t}-\ref{line:da_sddp_compute_bound}).

\paragraph{Forward Step.} For each stage $t \in [T]$, DA-SDDP solves the following approximation of Problem~\eqref{eq:dra_bellman} (or \eqref{eq:drr_bellman}), which we refer to as \textit{subproblem} and denote by $\Prob^l_t(x^l_{t-1}, \xi^l_t)$ (Line~\ref{line:da_sddp_solve_subproblem}): 
\begin{align} \label{eq:subproblem} 
\hat{Q}^l_t(x^l_{t-1}, \xi^l_t) :=
    \min_{(x_t, y_t) \in X_t(x^l_{t-1}, \xi^l_t)} \ & \Big\{f_t(x_t, y_t, \xi^l_t) + \phi^l_{t}(x_t)\Big\}, \quad t \in [T],
\end{align}

\noindent
where $\phi^l_{T}(\cdot) = 0$, and $x^l_{t-1}, t = 2, .., T,$ is an optimal stage-$(t-1)$ solution.  To simplify notation, we let $x^l_0:=x_0$ and $X_1(x^l_0, \xi^l_1) := X_1$. Function $\phi^l_t(x_t)$, for each $t \in [T-1]$, is a convex function that is constructed by $K^l_t$ cuts, and it serves as an under-approximation of the pessimistic and optimistic expected cost-to-go functions---$\Q^{RA}_{t+1}$ and $\Q^{RR}_{t+1}$---while solving DRO-MSIP and DRR-MSIP, respectively.
The approximating functions presented in Sect.~\ref{sec:approx_drr_msip} and \ref{sec:approx_dra_msip} are viable substitutes for the function $\phi^l_t$ when they are defined using the set of cuts $\{(\alpha^{i,k}_t, \beta^{i,k}_t)\}_{k \in [K^l_t]}$ available at iteration $l$ and stage $t$.

\paragraph{Backward Step.} 
For each $t = T, \dots, 2$, the algorithm solves relaxations of the subproblems, denoted by $\tilde{\Prob}_{t}^l(x^l_{t-1}, \om_t^i)$, and compute affine cuts $(\alpha_{t-1}^{i,l}, \beta_{t-1}^{i,l})$ for $i \in \N_t$ (Line~\ref{line:da_sddp_solve_relaxation}) using the information obtained by solving the relaxations.
These cuts provide a lower-bounding approximation of the value function $\hat{Q}^l_{t}$ such that
\begin{subequations}
\begin{align}
    \hat{Q}^l_{t}(x_{t-1}, \om_{t}^i) \geq (\alpha^{i,l}_{t-1})^\top x_{t-1} + \beta^{i,l}_{t-1}, \quad \forall x_{t-1} \in \{0, 1\}^{d_x}, i \in \N_{t}, \label{eq:cut_valid_plane} \\
    \hat{Q}^l_{t}(x^l_{t-1}, \om_{t}^i) = (\alpha^{i,l}_{t-1})^\top x^l_{t-1} + \beta^{i,l}_{t-1}. \label{eq:cut_supporting_plane}
\end{align}
\end{subequations}
It should be noted that the cut's validity is defined for $\hat{Q}^l_t$, yet this condition is sufficient to construct a lower-approximation for the optimistic or pessimistic expected cost-to-go function, since $\hat{Q}^l_t$ is always lower-bounding the exact value function $Q^{RR}_t$ or $Q^{RA}_t$, respectively.
Using the cuts $\{(\alpha_{t-1}^{i,l}, \beta_{t-1}^{i,l})\}_{i \in \N_t}$, the algorithm tightens the approximating function $\phi^l_{t-1}(x_{t-1})$ and improves the lower bound (Line~\ref{line:da_sddp_refine_approx}). In Line~\ref{line:da_sddp_compute_bound}, the algorithm computes the lower bound on the overall optimal objective value by solving the subproblem associated with the first stage.

The algorithm repeats these iterations with the forward and backward steps until one of predetermined stopping conditions is satisfied. These conditions can include a maximum number of iterations, a limit on elapsed time, or convergence of the lower bound.

There are various ways of generating a cut $(\alpha^{i,l}_{t-1}, \beta^{i,l}_{t-1}), i \in \N_t,$ which is a supporting hyperplane of the epigraph of $\hat{Q}^l_t(x_{t-1}, \om^i_t)$, intersecting at $x_{t-1}=x^l_{t-1}$, i.e., a cut satisfying both \eqref{eq:cut_valid_plane} and \eqref{eq:cut_supporting_plane}.
For example, by solving the subproblem to optimality, we can obtain an integer optimality cut, given a lower bound $L$ for the value function $\hat{Q}^l_t$, in the following form:
\begin{align*}
    \hat{Q}^l_{t}(x_{t-1}, \om_{t}^i) \geq \big( \hat{q}^l_t - L \big) \bigg(\sum_{i \in [d_x]} 2x^l_{t-1,i} x_{t-1,i} - x_{t-1,i} - x^l_{t-1,i}\bigg) + \hat{q}^l_t,
\end{align*}
where $\hat{q}^l_t := \hat{Q}^l_t(x^l_{t-1}, \om^i_t)$. As another example, consider a Benders cut obtained by solving a linear programming relaxation of the subproblem when solving a DRO-MSIP (or DRR-MSIP). This cut satisfies \eqref{eq:cut_valid_plane}, though not necessarily \eqref{eq:cut_supporting_plane}. However, we can derive a mixed-binary linear programming reformulation of the subproblem by adding binary variables replacing integer variables, use the hierarchy of relaxations (discussed in Sect.~\ref{sec:hierarchy_relaxations}) to solve the reformulation to optimality, and obtain a Benders cut that satisfies both \eqref{eq:cut_valid_plane} and \eqref{eq:cut_supporting_plane}.

\subsection{Finite Convergence}

Now, we show the finite convergence of DA-SDDP equipped with the convex approximations presented earlier. For DRR-MSIPs, we use DRR-SDDP-C and DRR-SDDP-R to denote variants of DA-SDDP with the cutting plane-based approximation $\phi^C_t$~\eqref{eq:drr_decomp} and the reformulation-based approximation $\phi^R_t$~\eqref{eq:drr_reform_approx}, respectively. We define a \textit{policy} by a collection of functions $\{\bar{x}_t(\xi_{[t]}),\bar{y}_t(\xi_{[t]})\}_{t \in [T]}$, where $\xi_{[t]} = (\xi_1, \dots, \xi_t)$, which serves as a decision rule given any scenario path $(\xi_1, \dots, \xi_T)$. A policy is \textit{optimal} for a DRR-MSIP if $(\bar{x}_t(\xi_{[t]}), \bar{y}_t(\xi_{[t]}))$ is optimal to the $t$-th stage problem~\eqref{eq:drr_bellman} (\eqref{eq:drr_model} for $t=1$) for $t \in [T]$ and all $\xi \in \Xi$.

\begin{theorem} \label{thm:finite_convergence_drr_cutting_plane}
    The forward step of the DRR-SDDP-C algorithm defines an optimal policy for a DRR-MSIP in a finite number of iterations of its while loop with probability one.
    Furthermore, each iteration of the while loop is executed in a finite time if there exists a finite-time algorithm for computing the cut coefficients~\eqref{eq:drr_decomp_coefficients}.
\end{theorem}

\begin{remark}
When the ambiguity sets are singletons, the finite convergence result presented in Theorem~\ref{thm:finite_convergence_drr_cutting_plane} applies to MSIPs, where the finite convergence result previously established by Zou et al.~\cite{Zou}. While their proof utilizes assumptions regarding the validity, tightness, and finiteness of cuts, the proof in this paper with binary state variables demonstrates that the convergence holds even without the finiteness assumption on cuts.
\end{remark}
\begin{theorem} \label{thm:finite_convergence_drr_reform}
    The forward step of the DRR-SDDP-R algorithm defines an optimal policy to a DRR-MSIP in a finite number of iterations of its while loop with probability one.
    Furthermore, each iteration of the while loop is executed in a finite time if the ambiguity set at every stage is defined by a polytope or a mixed-binary linear set.
\end{theorem}

Similarly, for DRO-MSIPs, we refer to the variants of DA-SDDP as DRO-SDDP-C and DRO-SDDP-R, utilizing the cutting plane-based approximation $\phi^S_t$~\eqref{eq:dra_sep_approx} and the reformulation-based approximation $\phi^D_t$~\eqref{eq:dra_reform_approx}, respectively.
\begin{theorem} \label{thm:finite_convergence_dra}
    The forward step of the DRO-SDDP-C algorithm provides an optimal policy for DRO-MSIP in a finite number of iterations of its while loop with probability one.
    Furthermore, each iteration of the while loop is executed in a finite time if there exists a finite-time algorithm for solving the distribution separation problem~\eqref{eq:dra_dsp}.
\end{theorem}

\section{Extensions to Multistage Stochastic Disjunctive Programs with Distributional Ambiguity} \label{sec:extensions_da_msdp}

In this section, we present extensions of the DA-SDDP algorithms to DRR- and DRO-MSDPs defined in Section \ref{sec:ProbForm}, under the following assumption: For $h\in H_t$, set $\{(x_t, y_t) \in \R^{d_x}_+ \times \R^{d_y}_+: A^h_t(\om_t) x_t + B^h_t(\om_t) y_t \geq b^h_t(\om_t) - C^h_t(\om_t) x_{t-1}\}$ is nonempty and compact for any $x_{t-1} \in \{0,1\}^{d_x}$ and $\om_t \in \Om_t$.
Also, its constraints include $x_t \geq 0$ and $x_t \leq 1$.

\subsection{DA-SDDP algorithms for DRR- and DRO-MSDPs}

Let us consider a set of cuts, $\{(\pi^k_t, \gamma^k_t)\}_{k \in [K^l_t]}$, for iteration $l$ and stage $t$, constructing an approximation of the pessimistic or optimistic expected cost-to-go function.
Then, with the same definition of $x_0$ and $\om_1$ as in Algorithm~\ref{alg:da_sddp}, the subproblem at iteration $l$ is given by
\begin{align} \label{eq:da_msdp_sub}
\fontsize{11}{13.2}
\begin{split}
    \hat{Q}^l_t(x_{t-1}, \om_t) = 
    \min_{(x_t, y_t) \in X_t(x_{t-1}, \om_t)} \Big\{
        f_t(x_t, y_t, \om_t) + \phi_t : \\
        \phi_t \geq (\pi^k_t)^\top x_t + \gamma_{t}^k, \ \ k \in [K^l_t]
    \Big\}
\end{split}
\end{align}
for $t\in[T]$ and $\om_t \in \Om_t$, where $\phi_T=K^l_T = 0$ and $X_t(x_{t-1},\om_t)$ is defined by a disjunctive set \eqref{eq:disjunctive_set}.
We define the feasible set of the foregoing subproblem by
\begin{multline}
    \D^l_t(x_{t-1}, \om_t) := \Big\{
    (x_t, y_t, \phi_t) \in \R_+^{d_x} \times \R_+^{d_y} \times \R_+ : \\
    \bigvee_{h \in H_t} \Big(\phi_t - (\pi^{k}_t)^\top x_t \geq \gamma^{k}_t, \ k \in [K^l_t],\ \\
    A^h_t(\om_t) x_t + B^h_t(\om_t) y_t \geq b^h_t(\om_t) - C^h_t(\om_t) x_{t-1} \Big)
    \Big\},
\end{multline}
and derive the convex hull of $\D^l_t(x_{t-1}, \om_t)$ in the following proposition.

\begin{proposition} \label{prop:convex_hull_D}
    For any $x_{t-1} \in \{0,1\}^{d_x}$ and $\om_t \in \Om_t$, the convex hull of the set $\D^l_t(x_{t-1}, \om_t), t \in [T],$ is equivalent to the projection of polyhedral set $\tilde{\D}^{l}_t(x_{t-1}, \om_t)$ onto the $(x_t, y_t, \phi_t)$-space where $\tilde{\D}^{l}_t(x_{t-1}, \om_t)$ is given by
    \begin{equation*} 
    \fontsize{11}{13.2}
    \begin{split}
    \Bigg\{
        & \sum_{h\in H_t} \zeta^h_{t,0} = 1,
        \sum_{h\in H_t} \zeta^h_{t,1} - x_{t} = 0, \\
        & \sum_{h\in H_t} \zeta^h_{t,2} - y_t = 0,
        \sum_{h\in H_t} \zeta^h_{t,3} = x_{t-1}, \sum_{h\in H_t} \zeta^h_{t,4} - \phi_t = 0,
    \end{split}
    \end{equation*}
    \begin{equation*} 
    \fontsize{11}{13.2}
    \begin{split}
        & A^h_t(\om_t) \zeta^h_{t,1} + B^h_t(\om_t) \zeta^h_{t,2} + C^h_t(\om_t) \zeta^h_{t,3} - b^h_t(\om_t) \zeta^h_{t, 0} \geq 0, \quad h \in H_t, \\
        & \zeta^h_{t,4} - (\pi^k_t)^\top \zeta^h_{t,1} - \gamma^k_t \zeta^h_{t, 0} \geq 0, \quad h \in H_t, k \in K^l_t, \\
        & x_t \in \R^{d_x}_+, y_t \in \R^{d_y}_+, \phi_t \in \R_+, \\
        &\zeta^h_{t,0} \in \R_+, \zeta^h_{t,1} \in \R^{d_x}_+, \zeta^h_{t,2} \in \R^{d_y}_+, \zeta^h_{t, 3} \in \R^{d_x}_+, \zeta^h_{t, 4} \in \R_+, \ h \in H_t
        \Bigg\}.
    \end{split}
    \end{equation*}
\end{proposition}

Using Proposition~\ref{prop:convex_hull_D}, we also derive an extension of DA-SDDP for DRR- and DRO-MSDPs, namely DA-SDDP-DP. Its pseudocode is provided in Appendix~\ref{apx:da_sddp_dp_algorithm}. 
DA-SDDP-DP shares a similar structure to DA-SDDP, but note that it involves distinct subproblems and a special subroutine for adding cuts to the subproblems.
In particular, it solves the linear programming equivalents of subproblems~\eqref{eq:da_msdp_sub}, derived using Proposition~\ref{prop:convex_hull_D} and referred to as \textit{LP-subproblems}:
\begin{equation} \label{eq:da_msdp_sub_lp} 
    \min \Big\{ 
        f_t(x_t, y_t, \om_t) + \phi_t : 
            (x_t, y_t, \phi_t) \in \proj_{x_t, y_t, \phi_t}\big(\tilde{\D}^{l}_t(x_{t-1}, \om_t)\big) 
    \Big\},
\end{equation}
for $t \in [T]$ and $\phi_T=0$. We note that the DA-SDDP-DP algorithms for DRR- and DRO-MSDPs have the finite convergence if cuts $(\pi^l_{t-1}, \gamma^l_{t-1})$ are obtained as in the cutting plane-based algorithms for DRR- and DRO-MSIPs, respectively. 
The comprehensive description of the algorithm is provided in Appendix~\ref{apx:da_sddp_dp_algorithm}.

\subsection{Application of \texorpdfstring{Proposition~\ref{prop:convex_hull_D}}{Proposition 1} for solving DRR- and DRO-MSIPs using Hierarchical Relaxations} \label{sec:hierarchy_relaxations}

In this section, we present a hierarchy of relaxations ranging from linear relaxation to tight extended formulations for each stage to solve DRR- and DRO-MSIPs by applying Proposition~\ref{prop:convex_hull_D}. For the ease of exposition, let $y_t$ be continuous.
Then, the subproblem~\eqref{eq:subproblem} of DRR- and DRO-MSIPs for iteration $l$ and stage $t$ can be rewritten as
\begin{multline} \label{eq:subproblem_msip_rewrite_dp}
    \min \bigg\{f_t(x_t, y_t, \xi^l_t) + \phi^l_t:
    \Big(x_{t,j}=0 \vee x_{t,j}=1, j = 1, \dots, d_x\Big)  \\
    \bigwedge \Big(\phi_t - (\pi^k_t)^\top x_t \geq \gamma^k_t, \ k \in [K^l_t], \\
    A_t(\xi^l_t) x_t + B_t(\xi^l_t) y_t \geq b_t(\xi^l_t) - C_t(\xi^l_t) x_{t-1} \Big) \bigg\}.
\end{multline}
We use $\D^l_t(x_{t-1}, \xi^l_t)$ and $\D^{l,LP}_t(x_{t-1}, \xi^l_t)$ to denote the feasible region of the subproblem~\eqref{eq:subproblem_msip_rewrite_dp} and its linear programming relaxation, respectively. A relaxation of $\D^{l}_t(x_{t-1}, \xi^l_t)$ can be defined as follows:
$
    \D^{l,s}_t(x_{t-1}, \xi^l_t) := \D^{l,LP}_t(x_{t-1}, \xi^l_t) \cap \{x_{t,j} = 0 \vee x_{t,j} = 1, \ j \in [s]\}, \ \text{for } s \in [d_x].
$
It is easy to see that the set $\D^{l,s}_t(\cdot, \cdot)$ for $s=d_x$ is equivalent to the original set $\D^{l}_t(\cdot, \cdot)$. Moreover, $\D^{l,LP}_t(\cdot, \cdot) \supseteq \D^{l,1}_t(\cdot, \cdot) \supseteq \dots \supseteq \D^{l,d_x}_t(\cdot, \cdot) = \D^{l}_t(\cdot, \cdot),$
and thus $\conv(\D^{l,1}_t(\cdot, \cdot)) \supseteq \dots \supseteq \conv(\D^{l,d_x}_t(\cdot, \cdot)) = \conv(\D^{l}_t(\cdot, \cdot))$. This provides a hierarchy of relaxations of the feasible region $\D^{l}_t(\cdot, \cdot)$ of DRR- and DRO-MSIPs. The tight extended formulation of the convex hull of the relaxations can be obtained using Proposition~\ref{prop:convex_hull_D}.

\begin{proposition} \label{prop:hierarchy_relax}
    The convex hull of the set $\D^{l,s}_t(x_{t-1}, \om_t)$ is the projection of the following set onto the $(x_t,y_t,\phi_t)$-space for any $x_{t-1} \in \{0,1\}^{d_x}$ and $\om_t \in \Om_t$:
    \begin{equation} \label{eq:tight_form_hierarchy_relaxation}
    \fontsize{11}{13.2}
    \begin{split}
        \bigg\{
        & \sum_{h \in [|\J^s_t|]} \zeta^h_{t,0} = 1, \sum_{h \in [|\J^s_t|]} \zeta^h_{t,1} - x_t = 0, \\
        & \sum_{h \in [|\J^s_t|]} \zeta^h_{t,2} - y_t = 0, \sum_{h \in [|\J^s_t|]} \zeta^h_{t,3} = x_{t-1}, \sum_{h \in [|\J^s_t|]} \zeta^h_{t,4} - \phi_t = 0, \\
        & A_t(\om_t) \zeta^h_{t,1} + B_t(\om_t) \zeta^h_{t,2} + C_t(\om_t) \zeta^h_{t,3} \geq b_t(\om_t), \quad h \in [|\J^s_t|], \\
        & \zeta^h_{t,1,j} = 0, \quad j \in J^h_1,\ h \in [|\J^s_t|], \\
        & \zeta^h_{t,1,j} = \zeta^h_{t,0}, \quad j \in J^h_2, \ h \in [|\J^s_t|], \\
        & \zeta^h_{t,4} - (\pi^k_t)^\top \zeta^h_{t,1} - \gamma^k_t\zeta^h_{t,0} \geq 0, \quad h \in [|\J^s_t|], k \in [K^l_t], \\
        & x_t \in \R^{d_x}_+, y_t \in \R^{d_y}_+, \phi_t \in \R_+, \\
        & \zeta^h_{t,0} \in \R_+,\zeta^h_{t,1} \in \R^{d_x}_+,\zeta^h_{t,2} \in \R^{d_y}_+,\zeta^h_{t,3}\in \R^{d_x}_+,\zeta^h_{t,4}\in \R_+,\ h \in [|\J^s_t|]
        \bigg\},
    \end{split}
    \end{equation}
    where
    $\J^s_t:= \{(J^h_1, J^h_2) : h \in [|\J^s_t|]\}, s \in [d_x], t \in [T],$ be a set of all pairs of disjoint sets $(J_1, J_2)$ such that $J_1, J_2 \subseteq [d_x], J_1 \cap J_2 = \emptyset,$ and $|J_1 \cup J_2| = s$. 
\end{proposition}

Proposition~\ref{prop:hierarchy_relax} provides the following relaxation that can be used to generate cuts in Line~\ref{line:da_sddp_solve_relaxation} of Algorithm~\ref{alg:da_sddp}: 
$\tilde{Q}^{l,s}_t(x^l_{t-1}, \xi^l_t) = \min \Big\{f_t(x_t, y_t, \xi^l_t) + \phi^l_t(x_t): (x_t, y_t) \in \proj_{x_t, y_t}(\D^{l,s}_t(x^l_{t-1}, \xi^l_t)) \Big\},$
where $\tilde{Q}^{l,1}_t(x^l_{t-1}, \xi^l_t) \leq \dots \leq \tilde{Q}^{l,d_x}_t(x^l_{t-1}, \xi^l_t) = \hat{Q}^{l}_t(x^l_{t-1}, \xi^l_t)$. For $s=d_x$, a Benders cut obtained by solving this relaxation is a supporting hyperplane satisfying \eqref{eq:cut_supporting_plane}. For the smaller values of $s$, the Benders cut does not necessarily support the value function $\hat{Q}^{l}_t(\cdot)$, but the relaxations are computationally easier to solve. The value of $s$ can be adjusted either before or during the execution of the algorithm to address this trade-off between the computational effort and the effectiveness of the cuts.

\begin{remark}
    The above relaxations can be readily extended to the case where $y_t$ are mixed-binary variables. Furthermore, if $y_t$ are mixed-integer, a hierarchy of relaxations can be derived by employing the binary expansion as discussed in Remark~\ref{rem:binary_expansion}.
\end{remark}

\section{Computational Tests} \label{sec:computational_results}

In this section, we present computational results for the DRR-SDDP-C, DRR-SDDP-R, DRO-SDDP-C, and DRO-SDDP-R algorithms. Recall that these are specific implementations of DA-SDDP, utilizing the approximations \eqref{eq:drr_decomp}, \eqref{eq:drr_reform_approx}, \eqref{eq:dra_sep_approx}, and \eqref{eq:dra_reform_approx}, respectively. We apply these algorithms to solve instances of multistage maximum flow and facility location interdiction problems with distributional ambiguity.
All the algorithms are implemented in Julia 1.8 where subproblems, coefficient-computing problems~\eqref{eq:drr_decomp_coefficients}, and distribution separation problem~\eqref{eq:dra_dsp} are solved using Gurobi 9.5 with an optimality tolerance of $10^{-4}$. We also integrate our implementation of the inner functionalities of our DA-SDDP algorithms with \texttt{SDDP.jl} \cite{SDDP.jl} package to be consistent with the research community, e.g., \cite{Philpott18}, thereby making it convenient for future computational and applied users of these algorithms. We conducted all tests on a machine equipped with an Intel Core i7 processor (3.8 GHz), utilizing a single thread, and 32 GB RAM.

In our implementations, all the algorithms generate a strengthened Benders cut and an integer optimality cut alternately during the backward steps. For details on the strengthened Benders cuts, we refer readers to Zou et al.~\cite{Zou}. For ambiguity set, we consider Wasserstein ambiguity set with the $l_1$ norm throughout all test instances. Note that in this case the distribution separation problem~\eqref{eq:dra_dsp} is a linear program and the subproblem~\eqref{eq:subproblem} in the DRR-SDDP-R algorithm is a mixed-binary linear program.

\subsection{MS-MFIP with Distributional Ambiguity} \label{sec:max-flow-results}
\blue{
Throughout this section, we consider DRR and DRO variants of the MS-MFIP formulation that is provided in Appendix.~\ref{apx:nip}.
In the subsequent sections, we present the results of a comparative analysis of the algorithms and demonstrate the significance of DRR and DRO for MS-MFIP. In addition, we extend the analysis to the case involving data corruption where the defender can deceive the interdictor by providing false information about critical arcs in a network.
}

\vspace{-0.5em}
\subsubsection{Instance generation and computational results}

\vspace{-0.5em}
Networks are randomly generated, following the method presented in Cormican et al.~\cite{Cormican}.
First, we place all nodes, excluding the source and sink nodes, in a grid pattern. Next, we establish connections between the leftmost and rightmost nodes in the grid to the source and sink nodes, respectively, using non-interdictable arcs with infinite capacity. 
Then, every pair of adjacent nodes in the grid is connected by an arc. 
Horizontal arcs are oriented from left to right, and vertical arcs, connecting the leftmost or rightmost nodes, are oriented from up to down. 
The orientations of the remaining arcs are randomly chosen. To avoid trivial solutions, e.g., removing all horizontal arcs in the same column, we set 80 percent of all arcs to be interdictable. 
Following the above procedure, we generate two distinct networks with different sizes. 
For the first network and the second network, we sample realizations of the random capacity of each arc uniformly distributed on $[30, 60]$ and $[20, 90]$, respectively, to construct the support $\Om_t$ of size $|\Om|$ for each stage $t$.
We set the Wasserstein ball size parameter $\epsilon$ to $30$. The interdiction budget for each stage is set to one to avoid a trivial solution where an interdiction solution in an early stage completely separates the source and sink nodes, leaving no arcs to remove in later stages.
\begin{wraptable}[23]{r}{0.5\textwidth}
\small
\vspace{-5pt}
\setlength{\tabcolsep}{4pt}
    \centering
    \caption{Details of DRR- and DRA-MFIP instances}\label{tab:instance_summary_mfip}
    \begin{tabular}{lrrrr}
    \topline
    Instance &
    $|N| \times |A|$ &
    $T$ &
    $|\Omega|$ &
    \#Scenario \\ \midline
    NI-1-3-5  & 37 x 73  & 3 & 5  & 25     \\
    NI-1-3-10 &          &   & 10 & 100    \\
    NI-1-3-15 &          &   & 15 & 225    \\ \cmidrule{3-5}
    NI-1-4-5  &          & 4 & 5  & 125    \\
    NI-1-4-10 &          &   & 10 & 1000   \\
    NI-1-4-15 &          &   & 15 & 3375   \\ \cmidrule{3-5}
    NI-1-5-5  &          & 5 & 5  & 625    \\
    NI-1-5-10 &          &   & 10 & 10000  \\
    NI-1-5-15 &          &   & 15 & 50625  \\ \cmidrule{3-5}
    NI-1-6-5  &          & 6 & 5  & 3125   \\
    NI-1-6-10 &          &   & 10 & 100000 \\
    NI-1-6-15 &          &   & 15 & 759375 \\ \midline
    NI-2-3-5  & 52 x 106 & 3 & 5  & 25     \\
    NI-2-3-10 &          &   & 10 & 100    \\
    NI-2-3-15 &          &   & 15 & 225    \\ \cmidrule{3-5}
    NI-2-4-5  &          & 4 & 5  & 125    \\
    NI-2-4-10 &          &   & 10 & 1000   \\
    NI-2-4-15 &          &   & 15 & 3375   \\ 
    \bottomline
    \end{tabular}
\end{wraptable}

Table~\ref{tab:instance_summary_mfip} summarizes the details of the test instances. Each row corresponds to a single instance that is labeled accordingly under the \textit{Instance} column.
The naming convention NI-$i$-$T$-$|\Omega|$ is used for an instance with the $i$th network among the two aforementioned networks, $T$ stages, and $|\Om|$ realizations per stage.
The labels $|N|\times|A|$ and \#Scenario denote the number of nodes and arcs of the network and the total number of scenario paths, respectively.
For termination conditions, we specify a time limit of 3 hours for all algorithms. Also, there is an early-termination condition based on the convergence of lower bound. In particular, the algorithm is stopped if the lower bound fails to improve for 100 consecutive iterations.

\begin{table}[tbp]
\centering
\caption{Performance comparison of algorithms for DRR- and DRA-MFIP instances}\label{tab:comp_performance_result_mfip}
\setlength{\tabcolsep}{3pt}
\begin{tabular}{lrrrrrrrr}
\topline
  &
  \thead{2}{DRA-SDDP-C} & 
  \thead{2}{DRA-SDDP-R} &
  \thead{2}{DRR-SDDP-C} & 
  \thead{2}{DRR-SDDP-R} \\ 
  \cmidrule(lr){2-3} \cmidrule(lr){4-5} \cmidrule(lr){6-7} \cmidrule(lr){8-9}
    \multicolumn{1}{c}{Instance} & 
    \thead{1}{LBound} & 
    \thead{1}{Time (s)} & 
    \thead{1}{LBound} & 
    \thead{1}{Time (s)} & 
    \thead{1}{LBound} & 
    \thead{1}{Time (s)} & 
    \thead{1}{LBound} & 
    \thead{1}{Time (s)} \\ \midline
NI-1-3-5  & 537.08  & 39.1     & 537.08  & 224.1    & 528.10  & 41.1     & 528.10  & 585.0    \\
NI-1-3-10 & 535.00  & 50.3     & 535.00  & 1653.2   & 527.05  & 60.3     & 527.05  & 2172.9   \\
NI-1-3-15 & 534.76  & 105.0    & 534.76  & 10800+  & 528.11  & 103.1    & 528.11  & 6774.8   \\
NI-1-4-5  & 655.58  & 169.1    & 655.58  & 4794.1   & 644.76  & 173.8    & 644.76  & 10800+  \\
NI-1-4-10 & 670.44  & 547.2    & 662.12  & 10800+  & 655.45  & 453.2    & 600.68  & 10800+  \\
NI-1-4-15 & 657.02  & 401.6    & 641.21  & 10800+  & 643.02  & 525.3    & 590.32  & 10800+  \\
NI-1-5-5  & 735.84  & 1019.2   & 729.85  & 10800+  & 727.05  & 1455.2   & 645.57  & 10800+  \\
NI-1-5-10 & 718.36  & 1443.9   & 676.38  & 10800+  & 701.54  & 1958.1   & 583.27  & 10800+  \\
NI-1-5-15 & 792.14  & 704.2    & 774.19  & 10800+  & 769.31  & 746.9    & 650.32  & 10800+  \\
NI-1-6-5  & 737.51  & 3058.8   & 705.42  & 10800+  & 725.69  & 3491.4   & 631.55  & 10800+  \\
NI-1-6-10 & 762.90  & 5066.5   & 672.96  & 10800+  & 744.84  & 6853.9   & 593.40  & 10800+  \\
NI-1-6-15 & 748.97  & 10800+  & 594.48  & 10800+  & 715.37  & 10800+  & 576.73  & 10800+  \\
NI-2-3-5  & 1018.51 & 54.2     & 1018.51 & 115.3    & 1013.16 & 61.7     & 1013.16 & 906.5    \\
NI-2-3-10 & 892.61  & 73.9     & 892.61  & 994.0    & 884.66  & 96.8     & 884.66  & 3543.9   \\
NI-2-3-15 & 903.28  & 148.8    & 903.28  & 263.7    & 891.10  & 171.7    & 891.10  & 10244.7  \\
NI-2-4-5  & 1181.45 & 410.0    & 1181.45 & 9978.1   & 1173.41 & 436.3    & 1136.69 & 10800+  \\
NI-2-4-10 & 1096.20 & 1426.5   & 1062.48 & 10800+  & 1086.77 & 1574.4   & 981.57  & 10800+  \\
NI-2-4-15 & 1105.07 & 3387.2   & 1027.92 & 10800+  & 1089.51 & 3553.0   & 824.64  & 10800+  \\ 
\bottomline
\end{tabular}
\end{table}

In Table~\ref{tab:comp_performance_result_mfip}, we report lower bounds and solution times in seconds for each algorithm, labeled LBound and Time (s), respectively. 
We may also obtain upper bounds computed statistically through a sampling approach using the policy after termination. However, to focus on the computational performance of the algorithms in terms of convergence while running, we will only report lower bounds and solution times here.
The results indicate that the DRO-SDDP-C algorithm provides better lower bounds and solution times than the DRO-SDDP-R algorithm for all instances. On average, the DRO-SDDP-C algorithm converges 17.2 times faster than the DRO-SDDP-R algorithm. This performance advantage increases to 26.4 times (on average) for eight instances where both algorithms provide the same lower bounds.
The results show that the DRO-SDDP-R algorithm's performance is more susceptible to $T$ and $|\Omega|$ than the DRO-SDDP-C algorithm. For example, as $|\Omega|$ increases from 5 to 15 for NI-1-3-5, the DRO-SDDP-R algorithm's solution time increases by 48.2 times, while the DRO-SDDP-C algorithm's solution time increases by 2.7 times. Similarly, when $T$ increases to 5 and 6 from 3 for NI-1-3-5, the DRO-SDDP-R algorithm fails to solve the 5-stage and 6-stage instances---NI-1-5-5 and NI-1-6-5---within the time limit, whereas the DRO-SDDP-C algorithm solves all instances within the time limit. This is mainly because the DRO-SDDP-R algorithm adds a significantly larger number of cuts ($|\Omega|^2$ cuts for every subproblem solved) for each iteration, which increases the solution times for subproblems.
Regarding DRR-MFIP, the DRR-SDDP-C algorithm outperforms the DRR-SDDP-R algorithm for all test instances. 
In terms of solution time, the DRR-SDDP-C algorithm converges, on average, 22.3 times faster than the DRR-SDDP-R algorithm, and this advantage increases to 41.3 times for seven instances where both the algorithms provide the same bounds.
This is because the DRR-SDDP-R algorithm solves larger subproblems that arise from \eqref{eq:drr_reform_approx_b}, incorporated by the linearization and the ambiguity set, respectively.

\subsubsection{Impact of DRO and DRR on MS-MFIP} \label{sec:comp_result_out_of_sample1}

To demonstrate the impact of DRO and DRR on MS-MFIP, we present the results from out-of-sample tests, which are conducted as follows. We first sample realizations of the capacity over the stages. Then, the DRO-MFIP and DRR-MFIP instances obtained over this sample ($\{\Om_t: t \in [T]\}$) are solved using the algorithms. The resulting subproblems along with the cuts generated by the algorithms define a policy for MS-MFIP, i.e., the decision rule that selects the set of arcs to remove given a realization of the capacity. We simulate the DRR and DRO policies using scenario paths of the capacity that are sampled independently from the realizations used in solving the problems.

\begin{figure}
    \centering
    \begin{subfigure}{0.49\textwidth}
        \centering
        \includegraphics[width=\textwidth]{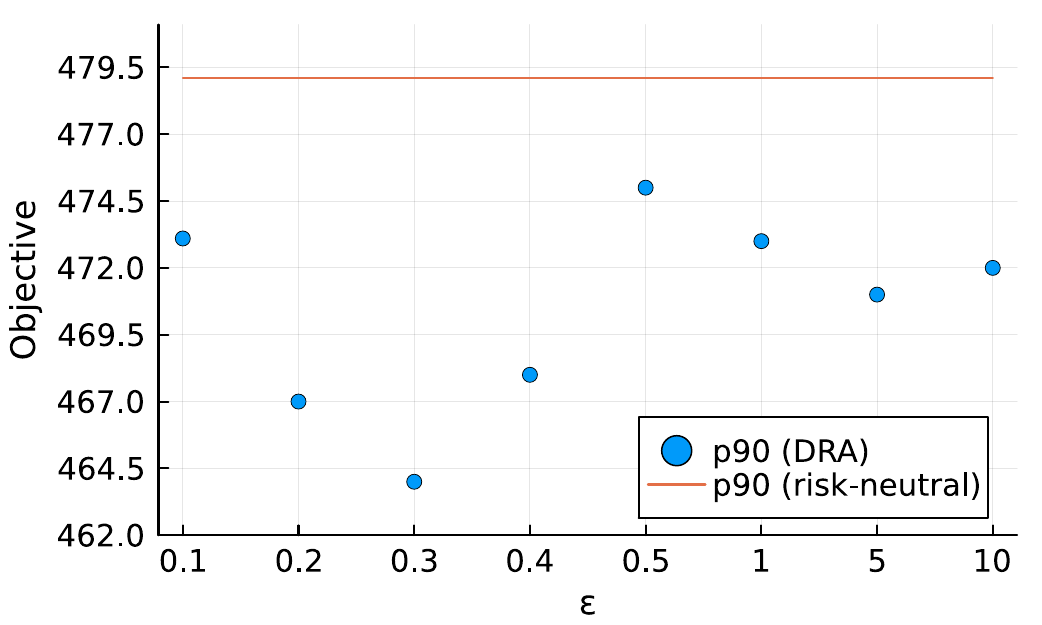}
        \caption{90th percentile (DRA-MFIP)}
        \label{fig:dra_mfip_1}
    \end{subfigure}
    \hfill
    \begin{subfigure}{0.49\textwidth}
        \centering
        \includegraphics[width=\textwidth]{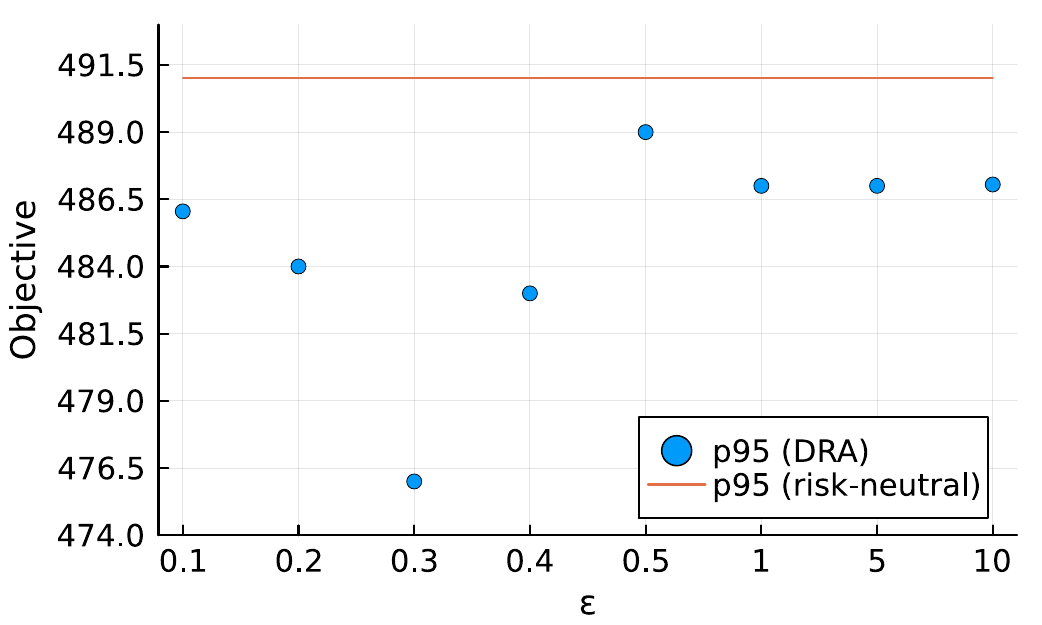}
        \caption{95th percentile (DRA-MFIP)}
        \label{fig:dra_mfip_2}
    \end{subfigure}
    \\
    \begin{subfigure}{0.49\textwidth}
        \centering
        \includegraphics[width=\textwidth]{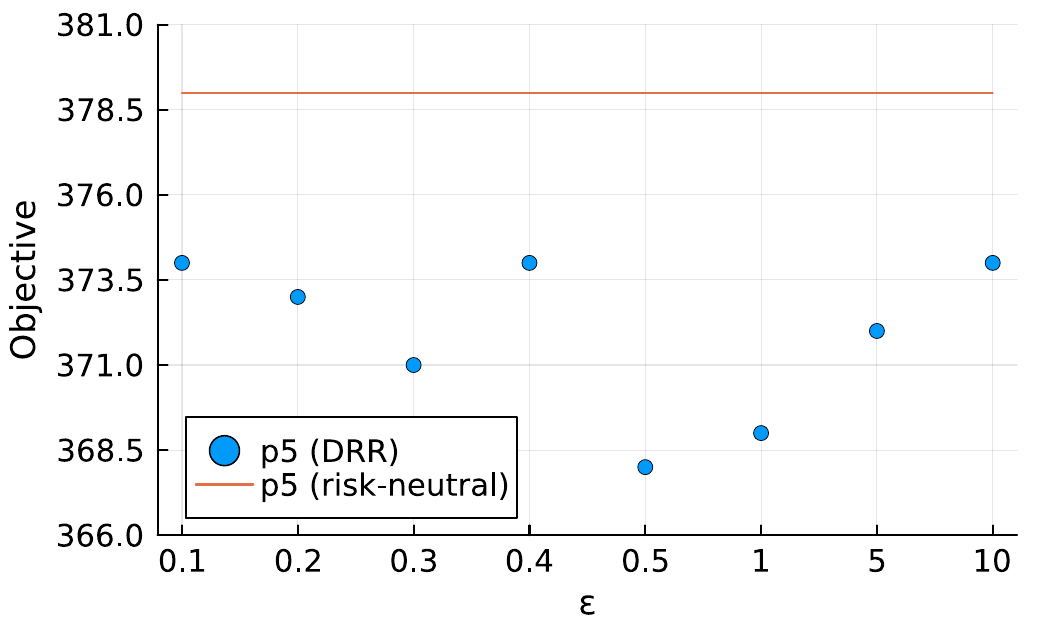}
        \caption{5th percentile (DRR-MFIP)}
        \label{fig:drr_mfip_1}
    \end{subfigure}
    \hfill
    \begin{subfigure}{0.49\textwidth}
        \centering
        \includegraphics[width=\textwidth]{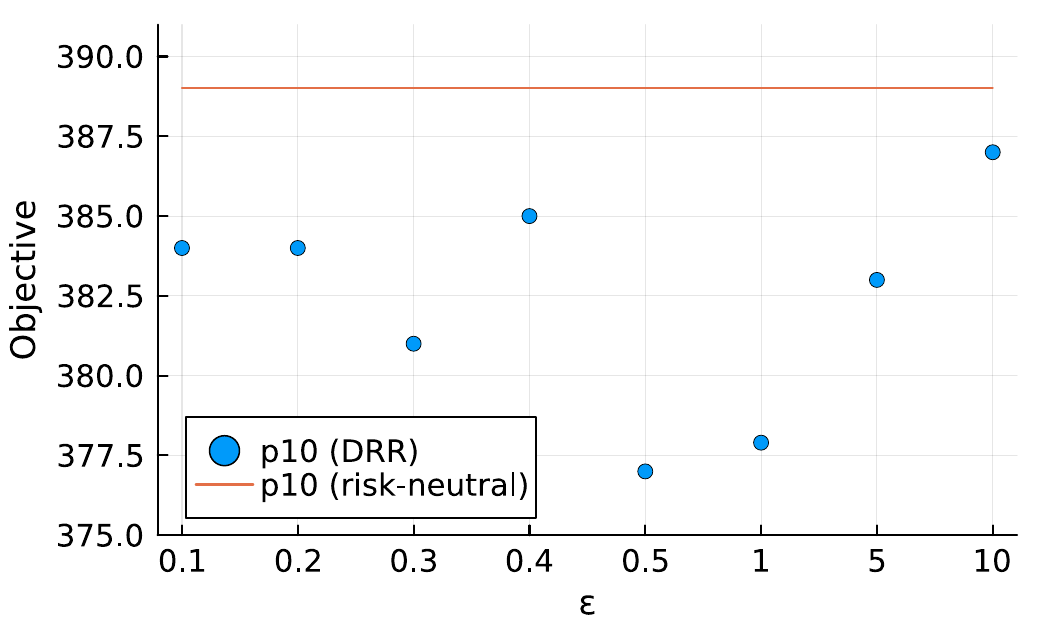}
        \caption{10th percentile (DRR-MFIP)}
        \label{fig:drr_mfip_2}
    \end{subfigure}
       \caption{Plots of percentiles from out-of-sample tests on DRA-MFIP (90th and 95th percentiles) and DRR-MFIP (5th and 10th percentiles)}
       \label{fig:dra_drr_mfip}
\end{figure}

Throughout the out-of-sample tests in this section, we consider the problem with 4 stages, 30 realizations per stage, 3000 independently-sampled scenario paths, and a network with 30 nodes and 60 arcs. All capacity realizations are sampled from a truncated normal distribution with a mean of $30$, a standard deviation of $5$, and values constrained to the interval $[10, 50]$. The policies are generated for the set of different Wasserstein ball size parameter $\epsilon$ belonging to $\{0, 0.1, 0.2, 0.3, 0.4, 0.5, 1, 5, 10\}$. 
We present the results from the simulations in Figure~\ref{fig:dra_drr_mfip}: Figures~\ref{fig:dra_mfip_1} and \ref{fig:dra_mfip_2} show the 90th and 95th percentiles of the objective values obtained by the DRO policies, and Figures~\ref{fig:drr_mfip_1} and \ref{fig:drr_mfip_2} show the 5th and 10th percentiles of the objective values obtained by the DRR policies. Each figure contains the result from the policy generated by SDDiP with $\epsilon=0$, i.e., the risk-neutral policy, as an orange horizontal line for comparison.

The DRO policies give the smaller 90th and 95th percentiles compared to the risk-neutral policy. As the ball size $\epsilon$ increases from $0.1$ to $0.3$, the corresponding policies yield the lower 90th and 95th percentiles, indicating the better out-of-sample performance. As the ball size $\epsilon$ increases further, to $0.5, 1, 5,$ and $10$, the overall performance of the policies drops, yet their 90th and 95th percentiles remain smaller than those of the risk-neutral policy. This demonstrates that the DRO policies achieve the robustness of interdiction solutions by incorporating conservatism over unfavorable realizations in $\Om_t, t \in [T]$. 
On the other hand, the DRR policies give the smaller 5th and 10th percentiles compared to the risk-neutral policy. 
This indicates that the DRR policies yield more effective interdiction solutions than the risk-neutral policy for certain scenario paths and reflect a more pessimistic perspective of the network user on the network performance. Consequently, this pessimistic view can identify the network vulnerabilities that are unnoticed by the risk-neutral policy.
The level of pessimism increases as $\epsilon$ increases from $0.1$ to $0.5$, except for $\epsilon=0.4$. However, as it continues to increase up to $\epsilon=10$, the pessimistic view on the performance diminishes.

\subsubsection{Significance of DRR framework for MS-MFIP with data corruption} \label{sec:comp_result_out_of_sample2}
In this test setup, we assume the network user has a capability of deceiving the interdictor by providing false information about critical arcs, i.e., most essential arcs to maintain the maximum flow. Specifically, the network user introduces uncertainty about the success of interdiction on the critical arcs before the interdictor makes a decision. In this adversarial setting, we consider a risk-receptive interdictor, i.e., an optimistic decision-maker, to mitigate the impact of the false information. The significance of DRR is demonstrated through out-of-sample tests we conduct as follows.

To simplify analysis, we focus on a two-stage case of MS-MFIP, where the interdiction decision is made only in the first stage, and the flow decision is made only in the second stage. That is, the interdictor's decision is made based on the corrupted sample data of the random vector in the second stage. Let $G = (N, A)$ represent the network. For random data generation, we generate samples of the capacity of each arc $a \in A$ from a normal distribution with mean $\mu_a$ and standard deviation $\sigma_a = \mu_a / 4$, where $\mu_a$ is drawn from a discrete uniform distribution in the interval $[20, 40]$. The network user manipulates these sample data by solving deterministic MS-MFIP on $G$ with true mean values, identifying the critical arcs in $A$, and introducing fake data where interdiction on these arcs fails, replacing an $\bar{\alpha}$ proportion of the sample data. The parameter $\bar{\alpha} \in [0, 1]$ denotes the contamination level. Throughout all tests, we fix the number of scenarios to 30 for consistency.

\begin{figure}
    \centering
    \includegraphics[width=0.7\textwidth]{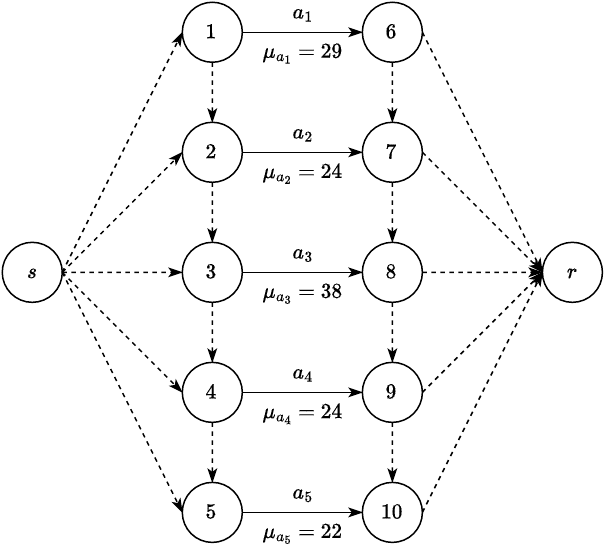}
    \caption{A network in a $5\times2$ grid shape.}
    \label{fig:network5x2}
\end{figure}
We first consider a simple network depicted in Figure~\ref{fig:network5x2}. The solid lines represent interdictable arcs, and the dotted lines represent non-interdictable arcs. The label on each arc represents its name and mean of capacity.
In this network, the most critical arc is $a_3$, followed by $a_1$ as the second most critical. The third most critical arc is either $a_2$ or $a_4$, as they share the same mean value. We assume the network user manipulates the data associated with arcs $a_1, a_2,$ and $a_3$.
Now we assess the impact of the corrupted data and DRR on optimal interdiction solutions by solving the DRR variant of MS-MFIP with Wasserstein ambiguity set. We also solve the DRO variant for comparative purposes. The contamination level $\bar{\alpha}$ varies in $\{0.2, 0.4, 0.6\}$, and the Wasserstein ball size parameter $\epsilon$ varies in $\{0, 10, 20, 40\}$. The interdiction budget $b$ is set to $2$. In Table~\ref{tab:network5x2-result}, we present the optimal interdiction solutions for these models.
\begin{table}
    \centering
    \caption{Optimal interdiction solutions for the DRR and DRO models with $\bar{\alpha} \in \{0.2, 0.4, 0.6\}$ and $\epsilon\in\{0,10,20,40\}$.}
    \label{tab:network5x2-result}
    \begin{tabular}{cccccccc}
    \topline
         & \thead{7}{$\epsilon$} \\ \cmidrule(lr){2-8}
         & \thead{1}{Risk-neutral} & \thead{3}{DRR} & \thead{3}{DRO} \\ \cmidrule(lr){2-2} \cmidrule(lr){3-5} \cmidrule(lr){6-8}
         \thead{1}{$\bar{\alpha}$} & \thead{1}{$0$} & \thead{1}{$10$} & \thead{1}{$20$} & \thead{1}{$40$} & \thead{1}{$10$} & \thead{1}{$20$} & \thead{1}{$40$} \\ \midline
         $0.2$ & $(a_3, a_4)$ & $(a_1, a_3)$ & $(a_1, a_3)$ & $(a_2, a_3)$ & $(a_3, a_4)$ & $(a_4, a_5)$ & $(a_4, a_5)$ \\
         $0.4$ & $(a_3, a_4)$ & $(a_1, a_3)$ & $(a_1, a_3)$ & $(a_2, a_3)$ & $(a_4, a_5)$ & $(a_4, a_5)$ & $(a_4, a_5)$ \\
         $0.6$ & $(a_4, a_5)$ & $(a_4, a_5)$ & $(a_1, a_3)$ & $(a_1, a_3)$ & $(a_4, a_5)$ & $(a_4, a_5)$ & $(a_4, a_5)$ \\
    \bottomline
    \end{tabular}
\end{table}
When the interdictor is risk-neutral, i.e., $\epsilon = 0$, solutions $(a_3, a_4)$ and $(a_4, a_5)$ are chosen. This demonstrates that the fake uncertainty of interdiction success motivates the interdictor to deviate from the optimal solution $(a_1, a_3)$ with the true mean values by choosing arcs $a_4$ and $a_5$, where the success of interdiction is guaranteed across all scenarios.
The risk-averse interdictor prioritizes the most conservative solution $(a_4, a_5)$, illustrating how DRO could be susceptible to data corruption in adversarial settings.
In contrast, the risk-receptive interdictor includes the critical arcs $a_1, a_2,$ and $a_3$ in their solutions in almost all cases. For $\bar{\alpha} = 0.6$, the optimal solution diverges from these critical arcs when $\epsilon=10$, selecting $(a_4, a_5)$, but it reverts to interdicting the critical arcs as $\epsilon$ increases to $20$ and $30$.
\begin{figure}
    \centering
    \includegraphics[width=0.75\textwidth]{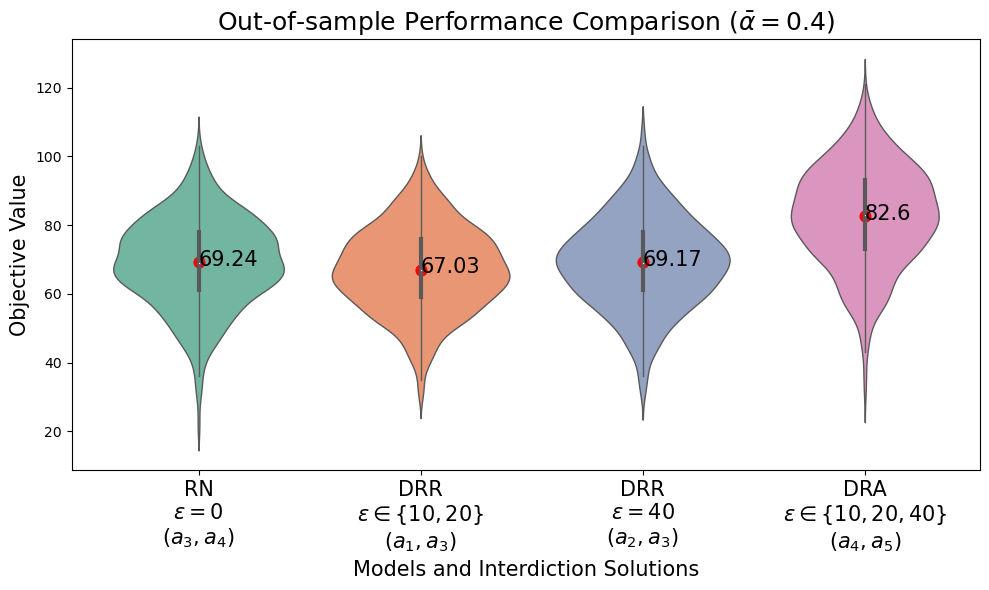}
    \caption{Comparison of interdiction solutions from different models in out-of-sample test. Network has a $5\times2$ grid shape, and contamination level $\bar{\alpha}$ is $0.4$.}
    \label{fig:oos-adversarial-1}
\end{figure}
Next, we evaluate the out-of-sample performance of these interdiction solutions. We randomly sample 1000 scenario paths from the clean distributions and solve the second-stage problem using the solutions provided in Table~\ref{tab:network5x2-result} for $\bar{\alpha}=0.4$. The plot of objective values for the sample paths is illustrated in Figure~\ref{fig:oos-adversarial-1}, with average values denoted by red marks along with their respective numbers. Since we are evaluating the interdictor's solutions, lower objective values indicate better performance. Comparing average performances, the DRR interdiction solutions, $(a_1, a_3)$ and $(a_2, a_3)$, outperform the other solutions. The DRO interdiction solution $(a_4, a_5)$ shows the most inferior out-of-sample performance among the four solutions with an average objective value that is $23.3$ percent higher than that of the DRR interdiction solution $(a_1, a_3)$.

\begin{figure}
    \centering
    \includegraphics[width=0.75\textwidth]{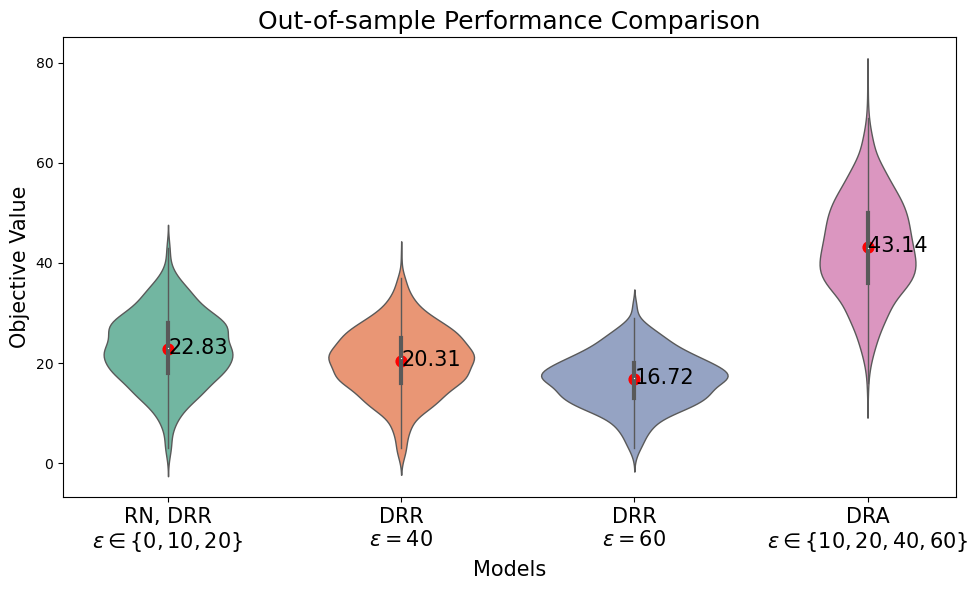}
    \caption{Comparison of interdiction solutions from different models in out-of-sample test. Network has a $6\times3$ grid shape, interdiction budget $b$ is $5$, and contamination level $\bar{\alpha}$ is $0.4$.}
    \label{fig:oos-adversarial-2}
\end{figure}
To assess the significance of DRR in a more complicated instance, we perform additional tests with the network in a $6\times3$ grid shape and the interdiction budget of $5$. We set $\bar{\alpha} = 0.4$, obtain the DRR and DRO solutions for $\epsilon\in\{0, 10, 20, 40\}$, and conduct the out-of-sample test with these solutions. The results are plotted in Figure~\ref{fig:oos-adversarial-2}. Consistent with our observations in the simpler instance, the out-of-sample performance of the DRR solutions improves as $\epsilon$ increases. In contrast, the DRO model produces identical solutions for all values of $\epsilon$, yielding the poorest average performance among the four solutions.

\subsection{Multistage facility location interdiction problem with distributional ambiguity} \label{sec:comp_result_flip}

\blue{
We evaluate the computational performance of the algorithms for MS-FLIP with distributional ambiguity, referred to as DRR-FLIP and DRO-FLIP. 
The detailed formulation of MS-FLIP is provided in Appendix~\ref{apx:nip}.
}

\subsubsection{Computational results}

To generate instances, we use the following method that is similar to the one used in \cite{Yu}. We first randomly sample $(L + M)$ points from a $100 \times 100$ grid and place demand points and facilities. For each demand point $l \in [L]$, we randomly choose $\mu_l$ from a uniform discrete distribution $[20, 40]$. Then, to construct the support $\Om_t$ of size $|\Omega|$ for each stage $t$, we randomly sample $|\Omega|$ realizations of the random demand for $l \in [L]$ from a truncated normal distribution, where the mean is $\mu_l$, the standard deviation is $\sigma_l=\mu_l / 4$, and the truncation interval is $[1, 60]$. 
For all test instances, we set the Wasserstein ball size $\epsilon$ to $10$ and set the interdiction budget for each stage to one, i.e., $r_t = 1, t \in [T]$.
The details of the test instances are given in Table~\ref{tab:instance_summary_flip}. Each instance, denoted by LI-$i$-$T$-$|\Om|$, involves the $i$th network out of two randomly generated networks, $T$ stages, and $|\Om|$ realizations per stage. For each row of the table, the labels $L\times M$ and \#Scenario denote the number of demand points and facilities and the number of total scenario paths, respectively.
The termination conditions are identical to those used for the DRR- and DRO-MFIP instances. 

\begin{wraptable}{r}{0.49\textwidth}
\small
\vspace{-18pt}
\setlength{\tabcolsep}{4pt}
\renewcommand{\arraystretch}{0.9}
\centering
\caption{Details of DRR- and DRA-FLIP instances}\label{tab:instance_summary_flip}
\begin{tabular}{lrrrr}
\topline
  Instance &
  $L \times M$ &
  $T$ &
  $|\Omega|$ &
  \#Scenario \\ \midline
LI-1-3-10 & 10 x 20  & 3 & 10 & 100 \\
LI-1-3-20 &          &   & 20 & 400 \\
LI-1-3-50 &          &   & 50 & 2500 \\ \cmidrule{3-5}
LI-1-4-10 &          & 4 & 10 & 1000 \\
LI-1-4-20 &          &   & 20 & 8000 \\
LI-1-4-50 &          &   & 50 & 125000 \\ \cmidrule{3-5}
LI-1-5-10 &          & 5 & 10 & 10000 \\
LI-1-5-20 &          &   & 20 & 160000 \\
LI-1-5-50 &          &   & 50 & 6250000 \\ \midline
LI-2-3-10 & 15 x 30  & 3 & 10 & 100 \\
LI-2-3-20 &          &   & 20 & 400 \\
LI-2-3-50 &          &   & 50 & 2500 \\ \cmidrule{3-5}
LI-2-4-10 &          & 4 & 10 & 1000 \\
LI-2-4-20 &          &   & 20 & 8000 \\
LI-2-4-50 &          &   & 50 & 125000 \\ \cmidrule{3-5}
LI-2-5-10 &          & 5 & 10 & 10000 \\
LI-2-5-20 &          &   & 20 & 160000 \\
LI-2-5-50 &          &   & 50 & 6250000 \\ 
\bottomline
\end{tabular}
\end{wraptable}

Table~\ref{tab:comp_performance_result_flip} presents the upper bounds and solution times in seconds obtained by each algorithm under the label UBound and Time (s). The numbers in each row of the table correspond to the result for a single instance. Note that smaller bounds are better since they are upper bounds. For DRO-FLIP, the DRO-SDDP-C algorithm provides the upper bounds better than the DRO-SDDP-R algorithm for all instances. Also, on average, the DRO-SDDP-C algorithm is 38.9 times faster than the DRO-SDDP-R algorithm, and this advantage increases to 48.4 times for 11 instances where both the algorithms produce the same bounds.
The performance of the DRO-SDDP-R algorithm is comparatively susceptible to the number of realizations per stage. For example, the DRO-SDDP-R algorithm takes 458.7 seconds to solve LI-$2$-$3$-$10$, but it fails to converge within the time limit when solving LI-$2$-$3$-$50$. The DRO-SDDP-C algorithm, however, converges for the both instances within the time limit. 
When comparing the results for DRR-FLIP, the DRR-SDDP-C algorithm provides better upper bounds than the DRR-SDDP-R algorithm for all the instances. 
In terms of solution time, the DRR-SDDP-C algorithm is, on average, 25.8 times faster than the DRR-SDDP-R algorithm for all instances, and 26.7 times faster for 13 instances where both the algorithms produce the same bounds.
As discussed in the previous section, this shows that the both DRO-SDDP-R and DRR-SDDP-R algorithms require more time to solve due to the larger subproblems resulting from the reformulation techniques.

\begin{table}[tb]
\centering
\caption{Performance comparison of algorithms for DRR- and DRA-FLIP instances}\label{tab:comp_performance_result_flip}
\setlength{\tabcolsep}{2.9pt}
\begin{tabular}{lrrrrrrrr}
\topline
  &
  \thead{2}{DRA-SDDP-C} & 
  \thead{2}{DRA-SDDP-R} &
  \thead{2}{DRR-SDDP-C} & 
  \thead{2}{DRR-SDDP-R} \\ 
  \cmidrule(lr){2-3} \cmidrule(lr){4-5} \cmidrule(lr){6-7} \cmidrule(lr){8-9}
    \multicolumn{1}{c}{Instance} & 
    \thead{1}{UBound} & 
    \thead{1}{Time (s)} & 
    \thead{1}{UBound} & 
    \thead{1}{Time (s)} & 
    \thead{1}{UBound} & 
    \thead{1}{Time (s)} & 
    \thead{1}{UBound} & 
    \thead{1}{Time (s)} \\ \midline
LI-1-3-10 & 431.29 & 29.1   & 431.29 & 169.7   & 437.28 & 28.4   & 437.28  & 109.1   \\
LI-1-3-20 & 327.32 & 49.8   & 327.32 & 1761.0  & 333.39 & 41.3   & 333.39  & 630.6   \\
LI-1-3-50 & 384.83 & 75.8   & 384.83 & 10800+ & 392.50 & 74.5   & 392.50  & 3373.3  \\
LI-1-4-10 & 489.80 & 53.2   & 489.80 & 1170.2  & 498.44 & 48.6   & 498.44  & 622.1   \\
LI-1-4-20 & 479.23 & 68.5   & 479.23 & 6511.9  & 491.03 & 78.7   & 491.03  & 1963.0  \\
LI-1-4-50 & 651.45 & 105.6  & 651.45 & 10800+ & 668.28 & 106.8  & 668.28  & 10800+ \\
LI-1-5-10 & 654.12 & 73.6   & 654.12 & 1744.3  & 670.83 & 74.5   & 670.83  & 1501.3  \\
LI-1-5-20 & 774.27 & 89.5   & 774.27 & 4853.6  & 792.99 & 80.3   & 792.99  & 4635.5  \\
LI-1-5-50 & 929.85 & 156.2  & 931.32 & 10800+ & 955.35 & 162.9  & 978.63  & 10800+ \\
LI-2-3-10 & 452.23 & 74.1   & 452.23 & 458.7   & 456.76 & 78.3   & 456.76  & 333.4   \\
LI-2-3-20 & 463.56 & 119.4  & 463.56 & 1894.1  & 468.56 & 112.7  & 468.56  & 772.8   \\
LI-2-3-50 & 489.47 & 459.4  & 518.37 & 10800+ & 495.10 & 453.1  & 495.10  & 10800+ \\
LI-2-4-10 & 500.09 & 369.0  & 500.09 & 10800+ & 506.22 & 378.5  & 506.22  & 4815.8  \\
LI-2-4-20 & 508.09 & 476.5  & 531.57 & 10800+ & 514.78 & 444.9  & 514.78  & 8376.8  \\
LI-2-4-50 & 556.50 & 761.3  & 721.03 & 10800+ & 564.69 & 692.1  & 1194.59 & 10800+ \\
LI-2-5-10 & 599.37 & 534.6  & 661.32 & 10800+ & 605.32 & 550.7  & 1009.53 & 10800+ \\
LI-2-5-20 & 816.66 & 879.1  & 904.14 & 10800+ & 828.21 & 1026.6 & 1339.16 & 10800+ \\
LI-2-5-50 & 756.16 & 1732.8 & 949.12 & 10800+ & 767.99 & 1927.1 & 1756.89 & 10800+ \\ 
\bottomline
\end{tabular}
\end{table}

\section{Conclusion} \label{sec:conclusion}

We studied multistage stochastic integer and disjunctive programs under distributional ambiguity, considering the distributional risk-receptiveness and robustness in a decision making process. 
For distributionally risk-receptive multistage stochastic integer programs (DRR-MSIPs) without and with decision dependent uncertainty, we presented new classes of cutting planes and reformulation-based approaches to derive convex approximations of the optimistic expected cost-to-go functions. For distributionally robust multistage stochastic integer programs (DRO-MSIPs), we presented a cutting plane-based and reformulation-based approximations of the pessimistic expected cost-to-go functions.
We developed algorithms using these approximations and provided their finite convergence analysis.
Furthermore, we extended the algorithms for distributionally risk-receptive and distributionally robust multistage stochastic disjunctive programs (DRR- and DRO-MSDPs) and then presented applications of them to solve DRR- and DRO-MSIPs using a hierarchy of relaxations. We compared the algorithms for DRR- and DRO-MSIPs by solving multistage stochastic network interdiction problems under distributional ambiguity that are multi-round attacker-defender games and have not been addressed in the literature. 
The computational results show that the algorithms using the cutting plane-based approximations outperform the algorithms using the reformulation-based approximations by 26.1 times, on average, in terms of solution time until convergence.
In addition, we conducted out-of-sample tests, and their results demonstrate that the DRA policies provide robust decision rules for uncertainty, while the DRR policies may reveal the network vulnerabilities that are overlooked by risk-neutral policies for uncertainty. In an adversarial setting, we showcased that the DRR policies can be used to mitigate the impact of data corruption.

\paragraph{Data Availability Statement}{The instances used for computational studies in this paper are available in ``Multistage-Interdiction-Problems'' folder at \url{https://github.com/Bansal-ORGroup/}.}

\bibliographystyle{spmpsci}
\bibliography{reference.bib}

\begin{thebibliography}{10}
\providecommand{\url}[1]{{#1}}
\providecommand{\urlprefix}{URL }
\expandafter\ifx\csname urlstyle\endcsname\relax
  \providecommand{\doi}[1]{DOI~\discretionary{}{}{}#1}\else
  \providecommand{\doi}{DOI~\discretionary{}{}{}\begingroup
  \urlstyle{rm}\Url}\fi

\bibitem{BalasDPBook}
Balas, E.: Disjunctive programming.
\newblock Springer (2018)

\bibitem{Bansal}
Bansal, M., Huang, K.L., Mehrotra, S.: Decomposition algorithms for two-stage
  distributionally robust mixed binary programs.
\newblock SIAM Journal on Optimization \textbf{28}(3), 2360--2383 (2018)

\bibitem{BansalJOGO}
Bansal, M., Zhang, Y.: Scenario-based cuts for structured two-stage stochastic
  and distributionally robust $p$-order conic mixed integer programs.
\newblock Journal of Global Optimization \textbf{81}(2), 391--433 (2021)

\bibitem{basciftci_ddd_2021}
Basciftci, B., Ahmed, S., Shen, S.: Distributionally robust facility location
  problem under decision-dependent stochastic demand.
\newblock European Journal of Operational Research \textbf{292}(2), 548--561
  (2021)

\bibitem{Bayraksan}
Bayraksan, G., Love, D.K.: Data-driven stochastic programming using
  phi-divergences, pp. 1--19.
\newblock INFORMS TutORials in Operations Research (2015)

\bibitem{bayraksan_bounds_2024}
Bayraksan, G., Maggioni, F., Faccini, D., Yang, M.: Bounds for multistage
  mixed-integer distributionally robust optimization.
\newblock SIAM Journal on Optimization \textbf{34}(1), 682--717 (2024)

\bibitem{BirgeNBD}
Birge, J.R.: Decomposition and partitioning methods for multistage stochastic
  linear programs.
\newblock Operations Research \textbf{33}(5), 989--1007 (1985)

\bibitem{Blanchet}
Blanchet, J., Li, J., Lin, S., Zhang, X.: Distributionally robust optimization
  and robust statistics.
\newblock arXiv preprint arXiv:2401.14655  (2024)

\bibitem{BlanchetTransport}
Blanchet, J., Li, J., Pelger, M., Zanotti, G.: Automatic outlier rectification
  via optimal transport.
\newblock arXiv preprint arXiv:2403.14067  (2024)

\bibitem{Brown}
Brown, G., Carlyle, M., Salmerón, J., Wood, K.: Defending critical
  infrastructure.
\newblock Interfaces \textbf{36}(6), 530--544 (2006)

\bibitem{CaoGao}
Cao, J., Gao, R.: Contextual decision-making under parametric uncertainty and
  data-driven optimistic optimization.
\newblock Optimization Online preprint:2021/10/8634  (2021)

\bibitem{Caroe}
Carøe, C.C., Schultz, R.: Dual decomposition in stochastic integer
  programming.
\newblock Operations Research Letters \textbf{24}(1), 37--45 (1999)

\bibitem{Church04}
Church, R.L., Scaparra, M.P., Middleton, R.S.: Identifying critical
  infrastructure: The median and covering facility interdiction problems.
\newblock Annals of the Association of American Geographers \textbf{94}(3),
  491--502 (2004)

\bibitem{Cormican}
Cormican, K.J., Morton, D.P., Wood, R.K.: Stochastic network interdiction.
\newblock Operations Research \textbf{46}(2), 184--197 (1998)

\bibitem{SDDP.jl}
Dowson, O., Kapelevich, L.: Sddp. jl: a julia package for stochastic dual
  dynamic programming.
\newblock INFORMS Journal on Computing \textbf{33}(1), 27--33 (2020)

\bibitem{Duchi}
Duchi, J.C., Glynn, P.W., Namkoong, H.: Statistics of robust optimization: A
  generalized empirical likelihood approach.
\newblock Mathematics of Operations Research \textbf{46}(3), 946--969 (2021)

\bibitem{Duque_ts_2022}
Duque, D., Mehrotra, S., Morton, D.P.: Distributionally robust two-stage
  stochastic programming.
\newblock SIAM Journal on Optimization \textbf{32}(3), 1499--1522 (2022)

\bibitem{Duque}
Duque, D., Morton, D.P.: Distributionally robust stochastic dual dynamic
  programming.
\newblock SIAM Journal on Optimization \textbf{30}(4), 2841--2865 (2020)

\bibitem{GangaBansal}
Gangammanavar, H., Bansal, M.: Stochastic decomposition method for two-stage
  distributionally robust linear optimization.
\newblock SIAM Journal on Optimization \textbf{32}(3), 1901--1930 (2022)

\bibitem{Gao}
Gao, R., Kleywegt, A.: Distributionally robust stochastic optimization with
  wasserstein distance.
\newblock Mathematics of Operations Research \textbf{48}(2), 603--655 (2023)

\bibitem{Gotoh}
Gotoh, J.y., Kim, M.J., Lim, A.E.: A data-driven approach to beating saa out of
  sample.
\newblock Operations Research  (2023)

\bibitem{Janjarassuk}
Janjarassuk, U., Linderoth, J.: Reformulation and sampling to solve a
  stochastic network interdiction problem.
\newblock Networks \textbf{52}(3), 120--132 (2008)

\bibitem{Jiang}
Jiang, N., Xie, W.: Distributionally favorable optimization: A framework for
  data-driven decision-making with endogenous outliers.
\newblock SIAM Journal on Optimization \textbf{34}(1), 419--458 (2024)

\bibitem{Kang}
Kang, S., Bansal, M.: Distributionally risk-receptive and risk-averse network
  interdiction problems with general ambiguity set.
\newblock Networks \textbf{81}(1), 3--22 (2023)

\bibitem{Kosmas}
Kosmas, D., Sharkey, T.C., Mitchell, J.E., Maass, K.L., Martin, L.:
  Multi-period max flow network interdiction with restructuring for disrupting
  domestic sex trafficking networks.
\newblock Annals of Operations Research pp. 1--64 (2022)

\bibitem{Lan}
Lan, G.: Complexity of stochastic dual dynamic programming.
\newblock Mathematical Programming \textbf{191}(2), 717--754 (2022)

\bibitem{Lulli}
Lulli, G., Sen, S.: A branch-and-price algorithm for multistage stochastic
  integer programming with application to stochastic batch-sizing problems.
\newblock Management Science \textbf{50}(6), 786--796 (2004)

\bibitem{Luo}
Luo, F., Mehrotra, S.: Distributionally robust optimization with decision
  dependent ambiguity sets.
\newblock Optimization Letters \textbf{14}(8), 2565--2594 (2020)

\bibitem{Malaviya}
Malaviya, A., Rainwater, C., Sharkey, T.: Multi-period network interdiction
  problems with applications to city-level drug enforcement.
\newblock IIE Transactions \textbf{44}(5), 368--380 (2012)

\bibitem{Nakao}
Nakao, H.: Distributionally robust optimization in sequential decision making.
\newblock Ph.d. thesis, University of Michigan, Ann Arbor (2021)

\bibitem{NakaoMDP}
Nakao, H., Jiang, R., Shen, S.: Distributionally robust partially observable
  markov decision process with moment-based ambiguity.
\newblock SIAM Journal on Optimization \textbf{31}(1), 461--488 (2021)

\bibitem{NguyenDROLikelihood}
Nguyen, V.A., Shafieezadeh-Abadeh, S., Yue, M.C., Kuhn, D., Wiesemann, W.:
  Optimistic distributionally robust optimization for nonparametric likelihood
  approximation.
\newblock In: Advances in Neural Information Processing Systems, vol.~32.
  Curran Associates, Inc. (2019)

\bibitem{noyan_2022}
Noyan, N., Rudolf, G., Lejeune, M.: Distributionally robust optimization under
  a decision-dependent ambiguity set with applications to machine scheduling
  and humanitarian logistics.
\newblock INFORMS Journal on Computing \textbf{34}(2), 729--751 (2022)

\bibitem{Park}
Park, J., Bayraksan, G.: A multistage distributionally robust optimization
  approach to water allocation under climate uncertainty.
\newblock European Journal of Operational Research \textbf{306}(2), 849--871
  (2023)

\bibitem{Pereira}
Pereira, M.V., Pinto, L.M.: Multi-stage stochastic optimization applied to
  energy planning.
\newblock Mathematical Programming \textbf{52}(1), 359--375 (1991)

\bibitem{Philpott18}
Philpott, A.B., de~Matos, V.L., Kapelevich, L.: Distributionally robust {SDDP}.
\newblock Computational Management Science \textbf{15}(3), 431--454 (2018)

\bibitem{Royset}
Royset, J.O., Chen, L.L., Eckstrand, E.: Rockafellian relaxation in
  optimization under uncertainty: Asymptotically exact formulations.
\newblock arXiv preprint arXiv:2204.04762  (2023)

\bibitem{Sadana}
Sadana, U., Delage, E.: The value of randomized strategies in distributionally
  robust risk-averse network interdiction problems.
\newblock INFORMS Journal on Computing \textbf{35}(1), 216--232 (2023)

\bibitem{Scarf}
Scarf, H.: A min-max solution of an inventory problem, pp. 201--209.
\newblock Stanford University Press (1958)

\bibitem{Smith}
Smith, J.C., Song, Y.: A survey of network interdiction models and algorithms.
\newblock European Journal of Operational Research \textbf{283}(3), 797--811
  (2020)

\bibitem{SongTrustRegion}
Song, J., Zhao, C.: Optimistic distributionally robust policy optimization.
\newblock arXiv preprint arXiv:2006.07815  (2020)

\bibitem{Yu}
Yu, X., Shen, S.: Multistage distributionally robust mixed-integer programming
  with decision-dependent moment-based ambiguity sets.
\newblock Mathematical Programming \textbf{196}(1), 1025--1064 (2022)

\bibitem{Zhang}
Zhang, S., Sun, X.A.: Stochastic dual dynamic programming for multistage
  stochastic mixed-integer nonlinear optimization.
\newblock Mathematical Programming \textbf{196}(1), 935--985 (2022)

\bibitem{Zou}
Zou, J., Ahmed, S., Sun, X.A.: Stochastic dual dynamic integer programming.
\newblock Mathematical Programming \textbf{175}(1), 461--502 (2019)

\end{thebibliography}

\newpage
\begin{appendices}

\section{List of Major Abbreviations} \label{apx:abbr}

\begin{itemize}
    \item \textit{DRO/DRR}: Distributionally robust (or robustness)/Distributionally risk-receptive (or risk-receptiveness).
    \item \textit{MSLP/MSIP/MSDP}: Multistage stochastic linear/integer/disjunctive program.
    \item \textit{SDDP}: Stochastic dual dynamic programming.
    \item \textit{DA-SDDP}: Our customized SDDP that addresses distributional ambiguity.
    \item \textit{DRR-SDDP-C}: DA-SDDP for DRR-MSIPs with the cutting plane-based approximation \eqref{eq:drr_decomp}.
    \item \textit{DRR-SDDP-R}: DA-SDDP for DRR-MSIPs with the reformulation-based approximation \eqref{eq:drr_reform_approx}.
    \item \textit{DRO-SDDP-C}: DA-SDDP for DRO-MSIPs with the cutting plane-based approximation \eqref{eq:dra_sep_approx}.
    \item \textit{DRO-SDDP-R}: DA-SDDP for DRO-MSIPs with the reformulation-based approximation \eqref{eq:dra_reform_approx}.
    \item \textit{MS-MFIP}: Multistage stochastic maximum flow interdiction problem
    \item \textit{MS-FLIP}: Multistage stochastic facility location interdiction problem
\end{itemize}

\section{Formulations of the Network Interdiction Problems} \label{apx:nip}

We provide detailed formulations of the multistage NIPs used in our computational tests. We focus on the risk-neutral formulations, as their DRO and DRR variants can be readily derived from these by appropriately maximizing or minimizing with respect to the probability distribution within an ambiguity, as in \eqref{eq:dra_model} and \eqref{eq:drr_model}.

\subsection{Formulation of MS-MFIP}

Consider a directed and capacitated network, denoted by $G = (N, A)$ where $N$ is a set of nodes and $A$ is a set of directed arcs of the network. The interdictor's objective is to minimize the total flow from the source node $s$ to the sink node $r$ of the network $G$ by interdicting a subset of arcs in $A$. In contrast, the network user's objective is to maximize the total flow given the interdicted network. At each stage, both the interdictor and the network user make their decisions as follows: The interdictor removes a set of arcs from the network $G$ given an interdiction budget, and the network user finds a maximum flow after observing the interdictor's decision. It is assumed that after an arc is interdicted, the network user cannot use it till the end of the time horizon. 

Let $x_t$ and $y_t$ be the interdiction decision vector and the flow decision vector, respectively, for stage $t \in [T]$. An interdiction decision $x_{t, a}$, for each arc $a \in A$, is binary, i.e., $x_{t, a}=1$, if an interdiction occurs on $a \in A$, and $x_{t, a}=0$, otherwise. Each arc $a \in A$ is associated with the interdiction cost, denoted by $f_{t,a}$. The interdiction budget is denoted by $b_t$, for each stage $t \in [T]$. We assume that the capacity of arc $a \in A$ is uncertain and denoted by $c_{t,a}(\bom_t)$.
We denote the set of the outgoing arcs and the set of the incoming arcs of node $n \in N$ by $\delta^+(n)$ and $\delta^-(n)$, respectively. For the brevity, we assume there exists a dummy arc from $r$ to $s$ in $A$ associated with infinite capacity and interdiction cost. Then, the bellman equation form of MS-MFIP is given by \eqref{eq:generic_msp_main} and \eqref{eq:generic_msp_bellman} where for each $t \in [T]$,
\begin{subequations} \label{eq:mmfip} \begin{align}
    Q_{t}(x_{t-1}, \bom_t) := \min \ & \psi_t(x_t, \bom_t) + \mathbb{E}_{P_{t+1}}\big[Q_{t+1} (x_t, \bom_{t+1})\big] \\ 
    \text{s.t.} \ 
        & x_t \geq x_{t-1} \label{eq:mmfip_a} \\
        & \sum_{a \in A} f_{t,a} x_{t, a} \leq b_t + \sum_{a \in A} f_{t,a} x_{t-1, a} \label{eq:mmfip_b} \\
        & x_t \in \{0,1\}^{|A|},
\end{align} \end{subequations}
and
\begin{subequations} \label{eq:max_flow} \begin{align}
\psi_t(x_t, \bom_t) := \max_{y_t \geq 0} \bigg\{y_{t, (r,s)} : &\sum_{a \in \delta^+(n)} y_{t, a} - \sum_{a \in \delta^-(n)} y_{t, a} = 0, \quad \forall n \in N, \label{eq:max_flow_con1} \\
    & y_{t,a} \leq c_{t,a} (\bom_t) (1-x_{t, a}), \quad \forall a \in A \label{eq:max_flow_con2} \bigg\}.
\end{align} \end{subequations}
In interdictor's problem~\eqref{eq:mmfip}, constraint~\eqref{eq:mmfip_a} ensures that the impact of interdiction on an arc remains till the end of the time horizon, and constraint~\eqref{eq:mmfip_b} restricts the total interdiction cost within the given budget $b_t$.
Given an interdiction solution $x_t$, the objective function of the interdictor's problem at stage $t \in [T]$, $\psi_t(x_t, \bom_t)$, is a value function that provides a maximum flow over the interdicted network. The function $\psi_t(x_t, \bom_t)$ is computed by solving the network user's problem~\eqref{eq:max_flow} where decision variable $y_{t,a}$ represents a flow on arc $a$ in $A$, constraints~\eqref{eq:max_flow_con1} enforce the flow balance on nodes in $N$, and constraints~\eqref{eq:max_flow_con2} restrict the capacity of arcs in $A$. Notice that because of constraints~\eqref{eq:max_flow_con2}, $y_{t,a}$ is restricted to be zero if the interdiction occurs on arc $a \in A$, i.e., if $x_{t,a} = 1$.

To solve each stage's problem during our tests, we first dualize the network user's problem~\eqref{eq:max_flow} and then reformulate the problem~\eqref{eq:mmfip} into a single-level minimization mixed-binary program.

\subsection{Formulation of MS-FLIP}

\blue{
The single-stage (deterministic) FLIP, also known as the $r$-interdiction median problem, is introduced by Church et al.~\cite{Church04}. 
The objective of the single-stage FLIP is to find a subset of $r$ facilities whose removal maximizes the network user's objective of minimizing the total weighted distance. MS-FLIP extends this problem to a multistage stochastic setting, where interdiction decisions are made sequentially over multiple time periods under demand uncertainty.
}

Let $L$ and $M$ be the number of demand points and facilities, respectively. At stage $t \in [T]$, let $x_{tm} \in \{0, 1\}$ be an interdiction decision variable, which equals 1, if the interdiction occurs on facility $m \in [M]$, or equals 0, otherwise. Variable $y_{tlm} \in \{0, 1\}$ denotes an assignment decision that represents whether demand point $l$ is assigned to facility $m$. 
We denote the random weighted distance between $l$ and $m$ by $c_{tlm}(\bom_t) = a_{tl}(\bom_t) d_{lm}$, where $d_{lm} > 0$ is the Euclidean distance between $l$ and $m$, and $a_{tl}(\bom_t)$ is an uncertain demand at point $l \in [L]$. The formulation of MS-FLIP is given by \eqref{eq:generic_msp_main} and \eqref{eq:generic_msp_bellman} with
\begin{subequations} \label{eq:mflip}
\begin{align}
    Q_t(x_{t-1}, \bom_t) := \max &\ \sum_{l \in [L], m \in [M]} c_{tlm}(\bom_t) y_{tlm} + \mathbb{E}_{P_{t+1}}\big[Q_{t+1}(x_t, \bom_{t+1})\big] \label{eq:mflip_a}\\
    \text{s.t.} 
    &\ \sum_{m \in [M]} y_{tlm} = 1, \quad \forall l \in [L], \label{eq:mflip_b}\\
    &\ x_t \geq x_{t-1}, \label{eq:mflip_c}\\
    &\ \sum_{m \in [M]} (x_{tm} - x_{t-1,m}) = r_t, \label{eq:mflip_d}\\
    &\ \sum_{n \in S_{lm}} y_{tln} \leq x_{tm}, \quad \forall l \in [L], m \in [M], \label{eq:mflip_e} \\
    &\ x_t \in \{0, 1\}^M, \ y_t \in \{0, 1\}^{L \times M}.
\end{align}
\end{subequations}
Here $S_{lm} := \{n \in [M]: d_{ln} > d_{lm}\}$ is the set of facilities that are farther than facility $m \in [M]$ is from demand point $l \in [L]$. The first term of the objective function~\eqref{eq:mflip_a} represents the total weighted distance of the assignment of demand points to non-interdicted facilities. Constraints~\eqref{eq:mflip_b} enforce each demand point to be assigned to a facility. Constraint~\eqref{eq:mflip_c} ensures that the facilities interdicted from the previous stages remain interdicted for the current stage. Constraint~\eqref{eq:mflip_d} ensures that the total number of interdictions occurred at the current stage equals to the budget $r_t$. 
Constraint~\eqref{eq:mflip_e}, for each $l \in [L]$ and $m \in [M]$, prevents the demand point $l$ from being assigned to facilities farther than the facility $m$, unless the facility $m$ is interdicted.
It should be noted that problem~\eqref{eq:mflip} is a single-level maximization problem, \blue{but not the max-min form of typical interdiction problems,} because of the assumption that each facility has enough capacity to cover all demand values and the demand point is always assigned to the closest facility through constraints~\eqref{eq:mflip_e}.

\clearpage
\section{Proofs} \label{apx:proofs}
\subsection{Proof of Theorem~\ref{thm:drr_cutting_plane_approx}}
\begin{proof}
For any $k \in [K^l_t]$ and $x_t \in \{0, 1\}^{d_x}$, it is satisfied that
\begin{equation*}
\begin{aligned}
    \min_{P_{t+1} \in \P_{t+1}}
    \sum_{i \in \N_{t+1}} & p^i_{t+1} \hat{Q}^l_{t+1} (x_{t}, \om^i_{t+1})
    \geq \min_{P_{t+1} \in \P_{t+1}} \sum_{i \in \N_{t+1}} p^i_{t+1}\bigg( (\alpha^{i,k}_t)^\top x_t + \beta^{i,k}_t \bigg) \\
    & = \min_{P_{t+1} \in \P_{t+1}} \sum_{i \in \N_{t+1}} p^i_{t+1}\bigg( (\alpha^{i,k}_t)^\top (x_t - \hat{x}_t) + (\alpha^{i,k}_t)^\top \hat{x}_t + \beta^{i,k}_t \bigg) \\
    & \geq \gamma^k_t + \min_{P_{t+1} \in \P_{t+1}} \sum_{i \in \N_{t+1}} p^i_{t+1} (\alpha^{i,k}_t)^\top (x_t - x^k_t) \\
    & \geq \gamma^k_t + \sum_{j \in [d_x]} \min_{P_{t+1} \in \P_{t+1}} \sum_{i \in \N_{t+1}} p^i_{t+1} \alpha^{i,k}_{t,j} (x_{t,j} - x^k_{t,j}).
\end{aligned}
\end{equation*}
In the last inequality, if $x^k_{t,j} = 0$, then $(x_{t, j} - x^k_{t,j}) \in \{0, 1\}$. It follows that we can fix the coefficient of the term $(x_{t, j} - x^k_{t,j})$ to 
$$
    \pi^k_{t,j} = \min_{P_{t+1} \in \P_{t+1}} \sum_{i \in \N_{t+1}} p^i_{t+1} \alpha^{i,k}_{t,j}.
$$
If $x^k_{t,j} = 1$, then $(x_{t, j} - x^k_{t,j}) \in \{0, -1\}$ and we can fix the coefficient to
$$
    \pi^k_{t,j} = \max_{P_{t+1} \in \P_{t+1}} \sum_{i \in \N_{t+1}} p^i_{t+1} \alpha^{i,k}_{t,j}.
$$
Fixing the coefficients for all $j \in d_x$ and $k \in [K^l_t]$, we have
\begin{equation} \label{eq:drr_decomp_inequality}
    \min_{P_{t+1} \in \P_{t+1}} \sum_{i \in \N_{t+1}} p^i_{t+1} \hat{Q}^l_{t+1}(x_t, \om^i_{t+1})
        \geq (\pi^k_t)^\top (x_t - x^k_t) + \gamma^k_t, \quad \forall k \in [K^l_t].
\end{equation}
Function $\phi^{l, C}_t(x_t)$ is constructed using the affine functions in the right-hand side of \eqref{eq:drr_decomp_inequality}, and thus it follows that $\phi^{l,C}_t(x_t) \leq \Q^{RR}_{t+1}(x_t)$.
\qed \end{proof}

\subsection{Proof of Proposition~\ref{prop:drr_continuous_strong_dual}}
\begin{proof}
Function $\Q^{RR}_t (x_{t-1}) = -\max_{P_{t} \in \P_{t}(x_{t-1})} \E_{P_t}[- Q^{RR}_t(x_{t-1}, \bom_t)].$
By applying Theorem~1 in \cite{Gao} to the maximization problem, the dual formulation is given by 
\begin{align*}
    \Q^{RR}_t (x_{t-1}) 
        = - \min_{\rho_t \geq 0, v_t^i} 
        \ &
        \epsilon_t(x_{t-1}) \rho_t + \sum_{i \in \N_t} \bar{p}^i_t v^i_t \\
        \text{s.t.} \ & \rho_t \norm{\om_t - \om^i_t} + v^i_t \geq - Q^{RR}_t(x_{t-1}, \om_t), \ \forall \om_t \in \Om_t, i \in \N_t.
\end{align*}
where $\bar{p}^i_t = 1/N_t$. Consequently, we obtain
\begin{align*}
    &\Q^{RR}_t (x_{t-1}) \\
    &= - \min_{\rho_t \geq 0} \bigg\{ \epsilon_t(x_{t-1}) \rho_t + \sum_{i \in \N_t} \bar{p}^i_t \max_{\om_t \in \Om_t} \Big\{ - \rho_t \norm{\om_t - \om^i_t} - Q^{RR}_t(x_{t-1}, \om_t) \Big\} \bigg\} \\
    &= \max_{\rho_t \geq 0} \bigg\{ -\epsilon_t(x_{t-1}) \rho_t + \sum_{i \in \N_t} \bar{p}^i_t \min_{\om_t \in \Om_t} \Big\{\rho_t \norm{\om_t - \om^i_t} + Q^{RR}_t(x_{t-1}, \om_t) \Big\} \bigg\}. \ \ \qed
\end{align*}
\end{proof}

\subsection{Proof of Theorem~\ref{thm:drr_continuous_cut}}
\begin{proof}
By Proposition~\ref{prop:drr_continuous_strong_dual}, we have
\begin{equation} \label{eq:Q_RR_drr_continuous}
    \Q^{RR}_t(x_{t-1}) = \max_{\rho_t \geq 0} \bigg\{ -\epsilon_t(x_{t-1}) \rho_t + \sum_{i \in \N_t} \bar{p}^i_t V^i_t(x_{t-1}, \rho_t) \Big\} \bigg\},
\end{equation}
where $\bar{p}^i_t = 1/N_t$ and
\begin{equation} \label{eq:V_i_t}
    V^i_t(x_{t-1}, \rho_t) := \min_{\om_t \in \Om_t} \Big\{ \rho_t \norm{\om_t - \om_t^i} + Q^{RR}_t(x_{t-1}, \om_t) \Big\} \ \text{for } i \in \N_t.
\end{equation}
Using the valid cuts with coefficients $\{(\pi^k_t, \gamma^k_t)\}_{k \in [K_t]}$ for $Q^{RR}_t(x_t)$, we obtain a lower-bounding approximation $\underline{V}^i_t$ of $V^i_t$ under assumptions \ref{amp:feasible_region} and \ref{amp:linear_objective_function}, where
\begin{subequations} \label{eq:V_bar_i_t}
\begin{align}
    \underline{V}^i_t(x_{t-1}, \rho_t) := \min_{\om_t, x_t} \ & \rho_t \norm{\om_t - \om_t^i} + c_t^\top x_t +  \theta_t \\
    \text{s.t.} \ 
    & A_t x_t - \om_t \geq - C_t x_{t-1} \label{eq:V_bar_i_t-b} \\
    & \om_t \in \Om_t = [l_t, u_t]^{d_\om} \label{eq:V_bar_i_t-c} \\
    & \theta_t \geq (\pi_t^k)^\top x_t + \gamma_t^k, \quad \forall k \in [K_t] \label{eq:V_bar_i_t-d} \\
    & x_t \in \X_t, \om_t \in \R^{d_\om}, \theta_t \in \R.
\end{align}
\end{subequations}
Then, a Lagrangian relaxation of \eqref{eq:V_bar_i_t} without integrality restrictions on $x_t$, where $\lambda^i_t, (\mu^i_t, \nu^i_t), \zeta^i_{tk}$ are dual multipliers associated with constraints \eqref{eq:V_bar_i_t-b}, \eqref{eq:V_bar_i_t-c}, \eqref{eq:V_bar_i_t-d}, respectively, is given by
\begin{equation} \label{eq:drr_continuous_v_lag}
\begin{split}
\min_{(x_t,\om_t,\theta_t) \in \R^{d_x + d_\om + 1}}
    & - (\lambda^i_t)^\top C_t x_{t-1} + (\mu^i_t)^\top l_t - (\nu^i_t)^\top u_t + \rho_t \norm{\om_t - \om_t^i} \\& + (\lambda^i_t - \mu^i_t + \nu^i_t)^\top \om_t + \bigg(c_t - \lambda^i_t A_t + \sum_{k \in [K_t]} \pi^k_t\bigg)^\top x_t \\& + \bigg(1 - \sum_{k \in [K_t]} \zeta^i_{tk}\bigg) \theta_t + \sum_{k \in [K_t]} \gamma^k_t \zeta^i_{tk}.
\end{split}
\end{equation}
Next, we prove that the Lagrangian dual is equivalent to
\begin{subequations} \label{eq:drr_continuous_v_dual}
\begin{align}
    \max \ &
    \begin{aligned}[t]
        - (\lambda^i_t)^\top C_t x_{t-1} + &(\mu^i_t)^\top l_t - (\nu^i_t)^\top u_t \\
        &+ (\lambda^i_t - \mu^i_t + \nu^i_t)^\top \om^i_t + \sum_{k \in [K_t]} \gamma^k_t \zeta_{tk}
    \end{aligned}\\
    \text{s.t.} \ 
    & A_t^\top \lambda^i_t  - \sum_{k \in [K_t]} \pi^k_t \zeta_{tk} = c_t \label{eq:drr_continuous_v_dual_b}\\
    & \norm{\lambda^i_t - \mu^i_t + \nu^i_t}_\ast \leq \rho_t \label{eq:drr_continuous_v_dual_c}\\
    & \sum_{k \in [K_t]} \zeta^i_{tk} = 1 \label{eq:drr_continuous_v_dual_d}\\
    & (\lambda^i_t, \mu^i_t, \nu^i_t, \zeta^i_t) \geq 0.
\end{align}
\end{subequations}
As a consequence, we can substitute $V^i_t$ for all scenarios $i \in \mathcal{N}_t$ in \eqref{eq:Q_RR_drr_continuous} with the value functions of these Lagrangian duals and obtain problem \eqref{eq:drr_continuous_dual_approx}. Since the Lagrangian duals are maximization problems, any feasible solutions of them yield a valid cut in the form of \eqref{eq:drr_continuous_cut}, and tightest cuts are attained by optimal solutions.

Constraints~\eqref{eq:drr_continuous_v_dual_b} and \eqref{eq:drr_continuous_v_dual_d} are straightforwardly derived to ensure $(c_t - \lambda A_t + \sum_{k \in [K_t]}\pi^k_t) = 0$, and $(1-\sum_{k \in [K_t]} \zeta_k) = 0$, thereby ensuring that \eqref{eq:drr_continuous_v_lag} is not unbounded because of unrestricted $x_t$ and $\theta_t$. Likewise, as shown below,
constraint~\eqref{eq:drr_continuous_v_dual_c} ensures that $$
\tau := \inf_{\om_t \in \R^{d_\om}}\{\rho_t \norm{\om_t - \om_t^i} + (\lambda^i_t - \mu^i_t + \nu^i_t)^\top \om_t\} > -\infty.
$$ 
For the ease of exposition, we use $(\lambda,\mu,\nu)$ instead of $(\lambda^i_t,\mu^i_t,\nu^i_t)$ in the following claim.
\begin{claim}
    Let $\norm{\cdot}_\ast$ be the dual norm of $\norm{\cdot}$.
    Given $\rho_t \geq 0$, we have $\tau > -\infty$ if and only if $(\rho_t - \norm{\lambda - \mu + \nu}_\ast) \geq 0$. Also, if $(\rho_t - \norm{\lambda - \mu + \nu}_\ast) \geq 0$, then $\tau = (\lambda - \mu + \nu)^\top \om^i_t$.
\end{claim}
\begin{proof}
    By letting $\om'_t = \om^i_t - \om_t$, the infimum $\tau$ can be rewritten as follows.
    \begin{subequations}
    \begin{align}
        \tau 
        &= (\lambda - \mu + \nu)^\top \om^i_t + \inf_{\om'_t \in \R^{d_\om}} \bigg\{ \rho_t \norm{\om'_t} - (\lambda - \mu + \nu)^\top \om'_t\bigg\}.
    \end{align}
    \end{subequations}
    Let $\tau^0 := \inf_{\om'_t \in \R^{d_\om}} \big\{ \rho_t \norm{\om'_t} - (\lambda - \mu + \nu)^\top \om'_t \big\}.$
    If $(\rho_t - \norm{\lambda - \mu + \nu}_\ast) \geq 0$, then we have
    $$\tau^0 \geq \inf_{\om'_t \in \R^{d_\om}} \Big\{ \norm{\om'_t} \cdot (\rho_t - \norm{\lambda - \mu + \nu}_\ast) \Big\} = 0.$$ The inequality holds by H{\"o}lder's inequality. The infimum $\tau^0 = 0$ since $\om'_t = 0$ is feasible. Therefore, $\tau = (\lambda - \mu + \nu)^\top \om^i_t$.

    Now we show the ``only-if'' direction using proof by contradiction. Suppose that 
    $\tau > -\infty$ and $(\rho_t - \norm{\lambda - \mu + \nu}_\ast) < 0$. This implies that $\tau^0$ is also bounded. By definition of dual norm, i.e., $\norm{\lambda - \mu + \nu}_\ast = \max_{\om_t \in \R^{d_\om}}\{ \om_t^\top (\lambda - \mu + \nu) : \norm{\om_t} \leq 1 \},$ there exists $\bar{\om}_t\in\R^{d_\om}$ such that 
    $$ 
    \rho_t < \bar{\om}_t^\top (\lambda - \mu + \nu) \text{ and } \norm{\bar{\om}_t} \leq 1.
    $$
    Thus, $\rho_t \norm{\bar{\om}_t} - \bar{\om}_t^\top(\lambda - \mu + \nu) < 0$. Considering scalar $r \to \infty$, we have $\rho_t \norm{r\bar{\om}_t} - (r\bar{\om}_t)^\top(\lambda - \mu + \nu) \to -\infty$. By taking $\om'_t = r \bar{\om}_t$, it follows that $\tau^0 = - \infty$, and $\tau = -\infty$.
\qed\end{proof}
\end{proof}

\subsection{Proof of Theorem~\ref{thm:finite_convergence_drr_cutting_plane}}
\begin{proof}
    
We first show that the finiteness of the while loop.
Let $\{\bar{x}^l_t(\xi_{[t]}),\bar{y}^l_t(\xi_{[t]}):t \in [T]\}$ be a policy that is defined by the forward step at iteration $l$ of the DRR-SDDP-C algorithm. Then, to achieve the optimality of the policy, it suffices to show that $\phi^{l, C}_t(\bar{x}^l_t(\xi_{[t]})) = \Q^{RR}_{t+1}(\bar{x}^l_t(\xi_{[t]}))$ for $t \in [T-1]$ and all $\xi \in \Xi$.

Let $\L$ denote the set of iterations where the policy is non-optimal. 
For each $t \in [T-1],$ let $l_t \in \L$ be the largest iteration such that
$\phi^{l, C}_t(\bar{x}^l_t(\xi_{[t]})) < \Q^{RR}_{t+1}(\bar{x}^l_t(\xi_{[t]}))$
for some $\xi \in \Xi$.
Also, given iteration $m$, let $\bar{X}^m_{t}, t \in [T-1],$ be the set of $\bar{x}^l_t(\xi_{[t]})$ for all future iterations $l \in \L\setminus [m]$ and $\xi \in \Xi$, i.e., $\bar{X}^m_{t} = \{\bar{x}^l_t(\xi_{[t]}):\xi \in \Xi, l > m, l \in \L\}$.

We first show that $l_{T-1}$ is finite with probability one.
At the forward step of iteration $l$, if we observe that $\phi^{l, C}_{T-1}(\bar{x}) < \Q^{RR}_T(\bar{x})$ for any $\bar{x} \in \bar{X}^1_{T-1}$, then for all future iterations $m > l$ we have
\begin{equation}
    \phi^{m, C}_{T-1}(\bar{x}) \geq \min_{P_{T} \in \P_T} \sum_{i \in \N_T} p^i_T \hat{Q}^l_T(\bar{x}, \om^i_T) = \Q^{RR}_T(\bar{x}).
\end{equation}
The above holds since the cut added at iteration $l$ (in the form of constraint \eqref{eq:drr_decomp}) satisfies both \eqref{eq:cut_valid_plane} and \eqref{eq:cut_supporting_plane}, and by definition $\hat{Q}^l_T$ is equivalent to $Q^{RR}_T$.
It follows that $\phi^{m,C}_{T-1}(\bar{x}) = \Q^{RR}_T(\bar{x})$ for any $m > l$.
Since $|\bar{X}^1_{T-1}| < \infty$ and every $\xi \in \Xi$ has a positive probability to be sampled, it holds with probability one that $\phi^{l,C}_{T-1}(\bar{x}) = \Q^{RR}_T(\bar{x})$ for all $\bar{x} \in \bar{X}^1_{T-1}$ after finitely many iterations.
This shows that $l_{T-1} < \infty$ with probability one.

Next, we show that $l_{T-2}$ is also finite with probability one. Suppose at iteration $l \geq l_{T-1}$ we observe that $\phi^{l, C}_{T-2}(\bar{x}) < \Q^{RR}_{T-1}(\bar{x})$ for any $\bar{x} \in \bar{X}^{l_{T-1}}_{T-2}$. Then, we have $\phi^{m,C}_{T-2}(\bar{x}) = \Q^{RR}_{T-1}(\bar{x})$ for all future iterations $m > l$ since
\begin{equation}\begin{aligned}
    \phi^{m,C}_{T-2}(\bar{x}) 
        &\geq \min_{P_{T-1} \in \P_{T-1}} \sum_{i \in \N_{T-1}} p^i_{T-1} \hat{Q}^l_{T-1}(\bar{x}, \om^i_{T-1}) \\
        &= \min_{P_{T-1} \in \P_{T-1}} \sum_{i \in \N_{T-1}} p^i_{T-1} Q^{RR}_{T-1}(\bar{x}, \om^i_{T-1}) = \Q^{RR}_{T-1}(\bar{x}).
\end{aligned}\end{equation}
The first equality holds since $\phi^{m,C}_{T-1}(\bar{x}) = \Q^{RR}_T(\bar{x})$ for $m > l_{T-1}$. So, we have $\phi^{m, C}_{T-2}(\bar{x}) = \Q^{RR}_{T-1}(\bar{x})$ for any future iteration $m > l$. Since $|\bar{X}^{l_{T-1}}_{T-2}| < \infty$ and every $\xi \in \Xi$ has a positive probability, there exists a finite iteration $l$ such that $\phi^{l, C}_{T-2}(\bar{x}) = \Q^{RR}_{T-1}(\bar{x})$ for all $\bar{x} \in \bar{X}^{l_{T-1}}_{T-2}$ with probability one, which implies $l_{T-2} < \infty$ with probability one. Similarly, we can prove by induction that $l_t,$ for all $t \in [T-1],$ are finite. This proves that $|\L| < \infty$ with probability one.

Now, we show that each iteration terminates in finite time.
This is followed by the facts that
the subproblems are bounded and thus solvable in a finite time using a branch-and-cut algorithm and Line~\ref{line:da_sddp_refine_approx} is executed in a finite time by assumption.

\qed \end{proof}

\subsection{Proof of Theorem~\ref{thm:finite_convergence_drr_reform}}
\begin{proof}
By following a similar argument to the proof of Theorem~\ref{thm:finite_convergence_drr_cutting_plane}, we can show that the number of the while-loop iterations required to define an optimal policy is finite with probability one and each while-loop iteration is executed in a finite time.
\qed \end{proof}

\subsection{Proof of Theorem~\ref{thm:finite_convergence_dra}}
\begin{proof}
Let $\{\bar{x}^l_t(\xi_{[t]}),\bar{y}^l_t(\xi_{[t]})\}_{t \in [T]}$ be a policy that is defined by the forward step at iteration $l$. To show its optimality, it suffices to show $\phi^{l, S}_t(\bar{x}^l_t(\xi_{[t]})) = \Q^{RA}_{t+1}(\bar{x}^l_t(\xi_{[t]}))$ for $t \in [T-1]$ and all $\xi \in \Xi$.
We can prove the statement by following a similar argument to the proof of Theorem~\ref{thm:finite_convergence_drr_cutting_plane}.
\qed \end{proof}

\subsection{Proof of Proposition~\ref{prop:convex_hull_D}}
\begin{proof}
Let $z_t = (x_t, y_t, \phi_t)$. For any $\bar{x}_{t-1} \in \{0,1\}^{d_x}$, stage $t$, and iteration $l$, the convex hull of $\D^l_t(\bar{x}_{t-1}, \om_t)$ is equivalent to the convex hull of $\F^l_t(\om_t) \cap \Eset(\bar{x}_{t-1})$ where
\begin{multline} \label{eq:F_set}
    \F^l_t(\om_t) := \bigg\{
    (x_{t-1}, z_t) \in \R_+^{d_x} \times \R_+^{d_x + d_y + 1}: \\
    \bigvee_{h \in H_t} \Big( \phi_t - (\pi^{k}_t)^\top x_t \geq \gamma^{k}_t, \ k \in [K^l_t], \\
    A^h_t(\om_t) x_t + B^h_t(\om_t) y_t + C^h_t(\om_t) x_{t-1} \geq b^h_t(\om_t) \Big)
    \bigg\},
\end{multline}
and $\Eset(\bar{x}_{t-1}):=\{(x_{t-1}, z_t):x_{t-1} = \bar{x}_{t-1}\}$.

We claim that $conv(\F^l_t(\om_t) \cap \Eset(\bar{x}_{t-1})) = conv(\F^l_t(\om_t)) \cap \Eset(\bar{x}_{t-1})$. 
Clearly, $conv(\F^l_t(\om_t) \cap \Eset(\bar{x}_{t-1})) \subseteq conv(\F^l_t(\om_t)) \cap \Eset(\bar{x}_{t-1})$.

To show that $conv(\F^l_t(\om_t) \cap \Eset(\bar{x}_{t-1})) \supseteq conv(\F^l_t(\om_t)) \cap \Eset(\bar{x}_{t-1})$, pick any point $(\bar{x}_{t-1}, \bar{z}_t) \in conv(\F^l_t(\om_t)) \cap \Eset(\bar{x}_{t-1})$. Then, there exist $(x_{t-1}^j, z_t^j) \in \F^l_t(\om_t)$ and $\lambda^j \in (0, 1], j = 1, \dots, J$ such that $\sum_{j \in [J]}\lambda^j = 1, \sum_{j \in [J]}\lambda^j x^j_{t-1} = \bar{x}_{t-1}, \sum_{j \in [J]}\lambda^j z^j_t = \bar{z}_t$. Since $\bar{x}_{t-1}$ is binary and $x_{t-1}^j$ belongs to $[0,1]^{d_x}$, this implies that $x^j_{t-1} = \bar{x}_{t-1}$ for all $j \in [J]$. Consequently, $(x^j_{t-1}, z^j_{t-1}) \in \F^l_t(\om_t) \cap \Eset(\bar{x}_{t-1}), j \in [J],$ and $(\bar{x}_{t-1}, \bar{z}_t) \in conv(\F^l_t(\om_t) \cap \Eset(\bar{x}_{t-1}))$. This completes the proof of the claim.

To obtain $conv(\F^l_t(\om_t))$, we first use Theorem~2.1 in \cite{BalasDPBook} and derive a tight extended formulation of $\F^l_t(\om_t)$, which is given by
\begin{equation} \label{eq:F_tight_form}
\begin{aligned}
\widehat{\F}^{l}_t(\om_t) := \Bigg\{
    & \sum_{h\in H_t} \zeta^h_{t,0} = 1,
    \sum_{h\in H_t} \zeta^h_{t,1} - x_{t} = 0, \sum_{h\in H_t} \zeta^h_{t,2} - y_t = 0,
        \\
    & \sum_{h\in H_t} \zeta^h_{t,3} - x_{t-1} = 0, \sum_{h\in H_t} \zeta^h_{t,4} - \phi_t = 0, \\
    & A^h_t(\om_t) \zeta^h_{t,1} + B^h_t(\om_t) \zeta^h_{t,2} + C^h_t(\om_t) \zeta^h_{t,3} - b^h_t(\om_t) \zeta^h_{t, 0} \geq 0, \quad h \in H_t, \\
    & \zeta^h_{t,4} - (\pi^k_t)^\top \zeta^h_{t,1} - \gamma^k_t \zeta^h_{t, 0} \geq 0, \quad h \in H_t, k \in K^l_t, \\
    & x_t \in \R^{d_x}_+, y_t \in \R^{d_y}_+, x_{t-1} \in \R_+^{d_x}, \phi_t \in \R_+, \\
    & \zeta^h_{t,0} \in \R_+, \zeta^h_{t,1} \in \R^{d_x}_+, \zeta^h_{t,2} \in \R^{d_y}_+, \zeta^h_{t, 3} \in \R^{d_x}_+, \zeta^h_{t, 4} \in \R_+, h \in H_t
    \Bigg\}.
\end{aligned}
\end{equation}
Since $\F^l_t(\om_t)$ is unbounded, the projection of the above formulation~\eqref{eq:F_tight_form} onto the $(x_{t-1}, z_t)$-space is the closed convex hull of $\F^l_t(\om_t)$. Consider $|H_t|$ polyhedra defined by disjunctive constraints for $h \in H_t$ in \eqref{eq:F_set}. They are nonempty and have identical recession cones. This implies that $conv(\F^l_t(\om_t))$ is a polyhedron, and thus the closed convex hull of $\F^l_t(\om_t)$ is equivalent to $conv(\F^l_t(\om_t))$. Hence, using the tight extended formulation~\eqref{eq:F_tight_form}, the convex hull of $\D^l_t(\bar{x}_{t-1}, \om_t)$ is given by $Proj_{x_{t-1}, z_t}(\widehat{\F}^l_t(\om_t)) \cap \Eset(\bar{x}_{t-1}),$ which is equivalent to the projection of the set $\tilde{\D}^{l}_t(\bar{x}_{t-1}, \om_t)$ onto the $z_t$-space.
\qed \end{proof}

\subsection{Proof of Proposition~\ref{prop:hierarchy_relax}}
\begin{proof}
For $s \in [d_x]$, the set $\D_t^{l,s}(x_{t-1}, \om_t), t \in [T],$ is given by the following disjunctive constraints:
\begin{equation}
\begin{aligned}
    \bigvee_{(J_1, J_2) \in \J^s_t}
    \bigg(
        x_{t,j} = 0,\ j \in J_1, \ x_{t,j} = 1, \ j \in J_2,\ 
        \phi_t - (\pi^k_t)^\top x_t \geq \gamma^k_t, \ k \in [K^l_t], \\
        A_t(\om_t) x_t + B_t(\om_t) y_t \geq b_t(\om_t) - C_t(\om_t) x_{t-1}
    \bigg),
\end{aligned}
\end{equation}
where $\J^s_t = \{(J_1, J_2): J_1, J_2 \subseteq [d_x], J_1 \cap J_2 = \emptyset, |J_1 \cup J_2| = s\}$. By applying Proposition~\ref{prop:convex_hull_D} to this disjunctive set, we obtain the tight extended formulation \eqref{eq:tight_form_hierarchy_relaxation}.
\qed \end{proof}

\section{Pseudocode of DA-SDDP-DP Algorithm} \label{apx:da_sddp_dp_algorithm}

\setcounter{ALC@unique}{0}
\begin{algorithm}
\caption{DA-SDDP-DP} \label{alg:da_sddp_dp}
\begin{algorithmic}[1]
    \STATE{\textbf{Initialize} 
    $l \gets 1$; $x_0 \gets$ initial state; $\om_1 \gets$ data at the first stage; $\Om_1:=\{\om_1\}$; $K^l_t \gets 0$ for $t = 1, \dots, T-1$;
    } \\
    \WHILE{(satisfying none of stopping conditions)} \label{line:da_sddp_dp_while} 
    \STATE{Sample a scenario path $\xi^l \in \Xi:=\Om_1 \times \dots \times \Om_T$} \label{line:da_sddp_dp_sample}
    \FOR[Forward Step]{$t \in [T]$} \label{line:da_sddp_dp_forward_for_t}
        \STATE{Solve $t$-stage LP-subproblem~\eqref{eq:da_msdp_sub} given $x_{t-1}=x^l_{t-1}$ and $\om_t = \xi^l_t$} \label{line:da_sddp_dp_solve_subproblem}
    \ENDFOR
    \FOR[Backward Step]{$t=T,\dots,2$\label{line:da_sddp_dp_backward_for_t}}
        \FOR{$i \in \N_t$ \label{line:da_sddp_dp_backward_for_i}}
            \STATE{Solve $t$-stage LP-subproblem~\eqref{eq:da_msdp_sub} given $x_{t-1}=x^l_{t-1}$ and $\om_t = \om^i_t$
            and obtain cut $(\sigma_{t-1,1}^{i,l}, \sigma_{t-1,0}^{i,l})$} \label{line:da_sddp_dp_solve_relaxation}
        \ENDFOR
        \STATE{Add cuts $(\pi^{l}_{t-1}, \gamma^{l}_{t-1})$ to $(t-1)$-stage LP-subproblem~\eqref{eq:da_msdp_sub} by using cuts $(\sigma_{t-1,1}^{i,l}, \sigma_{t-1,0}^{i,l}), i \in \N_t$} \label{line:da_sddp_dp_tighten_approx}
        \STATE{$K^l_{t-1} \gets K^l_{t-1} + 1$} \label{line:da_sddp_dp_add_K}
    \ENDFOR
    \STATE{Solve LP-subproblem~\eqref{eq:da_msdp_sub} for $t=1$ to obtain the bound $LB$} \label{line:da_sddp_dp_compute_bound}
    \STATE{$K^{l+1}_t \gets K^l_t$ for $t =1, \dots, T-1$}
    \STATE{$l \gets l + 1$}
    \ENDWHILE
    \RETURN LP-subproblems, $LB$ \label{line:da_sddp_dp_return}
\end{algorithmic}
\end{algorithm}

In this algorithm, let $x^l_t$ be an optimal solution obtained by solving the LP-subproblem~\eqref{eq:da_msdp_sub_lp} during a forward step (Line~\ref{line:da_sddp_dp_solve_subproblem}) for iteration $l$ and stage $t$.
In a backward step (Line~\ref{line:da_sddp_dp_solve_relaxation}), the algorithm solves the LP-subproblem~\eqref{eq:da_msdp_sub_lp} given $x^l_t$ and $\om^i_t$ to obtain a cut in the form of Benders cut, $(\sigma^{i,l}_{t-1,1}, \sigma^{i,l}_{t-1,0})$, where $\sigma^{i,l}_{t-1,1}$ and $\sigma^{i,l}_{t-1,0}$ are the optimal dual multipliers, associated with the constraints $\sum_{h\in H_t} \zeta^h_{t,0}=1$ and $\sum_{h \in H_t} \zeta^h_{t,3} = x^l_{t-1}$, respectively.
These cuts $(\sigma^{i,l}_{t-1,1}, \sigma^{i,l}_{t-1,0})$ for $i \in \N_t$ are used to derive a cut $(\pi^{l}_{t-1}, \gamma^l_{t-1})$; the cut $(\pi^{l}_{t-1}, \gamma^l_{t-1})$ takes the form of \eqref{eq:drr_decomp}, if we solve a DRR-MSDP, and it takes the form of \eqref{eq:dra_sep_approx}, if we solve a DRO-MSDP. Then, its copies for $h \in H_t$ are added to the $(t-1)$-stage LP-subproblem (Line~\ref{line:da_sddp_dp_tighten_approx}). In Line~\ref{line:da_sddp_dp_compute_bound}, it computes the lower bound by solving the first-stage LP-subproblem. DA-SDDP-DP repeats this procedure until a stopping condition is met. Both the DRR- and DRO-SDDP-DP algorithms, defined as Algorithm~\ref{alg:da_sddp_dp} with cut $(\pi^l_{t-1}, \gamma^l_{t-1})$ obtained as above for DRR-MSDPs and DRO-MSDPs, respectively, have the finite convergence that can be proved using Theorems~\ref{thm:finite_convergence_drr_cutting_plane} and \ref{thm:finite_convergence_dra}, respectively.

\begin{remark}
    In the implementation of Algorithm \ref{alg:da_sddp_dp}, we can establish the LP-subproblems once and reuse them in each iteration, without the need for repeated construction, by adding constraints $\zeta^h_{t,4} - (\pi^k_t)^\top \zeta^h_{t,1} - \gamma^k_t \zeta^h_{t,0}$ to the LP-subproblems as needed (in Line~\ref{line:da_sddp_dp_tighten_approx}).
\end{remark}

\end{appendices}

\end{document}